\documentclass[a4paper,reqno,11pt,oneside]{amsart}
\usepackage{amsmath,geometry}
\usepackage{color}
\usepackage{latexsym}
\usepackage{amsmath}
\usepackage{amsfonts}
\usepackage{amssymb}
\usepackage{graphicx}
\usepackage{esint}
\usepackage{slashed}

\numberwithin{equation}{section}

\newtheorem{theorem}{Theorem}[section]

\newtheorem{proposition}[theorem]{Proposition}
\newtheorem{lemma}[theorem]{Lemma}
\newtheorem{remark}[theorem]{Remark}
\newtheorem{corollary}[theorem]{Corollary}

\newcommand{\R}{\mathbb{R}}

\title{The Green function for $p-$Laplace operators}

\author{Sabina Angeloni}
\address{Sabina Angeloni, Dipartimento di Matematica e Fisica, Universit\`a degli Studi Roma Tre, Largo S.~Leonardo Murialdo 1, Roma 00146, Italy.}
\email{sabina.angeloni@uniroma3.it}

\author{Pierpaolo Esposito}
\address{Pierpaolo Esposito, Dipartimento di Matematica e Fisica, Universit\`a degli Studi Roma Tre, Largo S.~Leonardo Murialdo 1, Roma 00146, Italy.}
\email{esposito@mat.uniroma3.it}

\begin{document}

\begin{abstract}
On a bounded domain $\Omega \subset \R^N$, $N\geq 2$, we consider existence, uniqueness and ``regularity" issues for the Green function $G_\lambda$ of the quasi-linear operator $u \to -\Delta_p u-\lambda |u|^{p-2}u$ with $1<p \leq N$, homogeneous Dirichlet boundary condition and $\lambda<\lambda_1$, where $\lambda_1>0$ is the first eigenvalue of $-\Delta_p$.  
\end{abstract}

\maketitle

\section{Introduction}
Given $1<p\leq N$ and a bounded domain $\Omega \subset \R^N$, $N \geq 2$,  for $x_0 \in \Omega$ we are interested to nonnegative solutions $G_\lambda$ of
$$-\Delta_p G - \lambda G^{p-1}= 0 \qquad \text{in } \Omega \setminus \{x_0\},$$
where $\Delta_p{(\cdot)}=\text{div}\bigl(|\nabla(\cdot)|^{p-2} \nabla(\cdot)\bigr)$ is the $p$-Laplace operator and $\lambda <\lambda_1$. Here $G_\lambda \in W_{\text{loc}}^{1,p}(\Omega \setminus \{x_0\})$ and $\lambda_1$ is the first eigenvalue of $-\Delta_p$ given by
$$\lambda_1= \inf_{u \in W_0^{1,p}(\Omega) \setminus \{0\}}  \frac{\int_\Omega |\nabla u|^p}{\int_\Omega |u|^p}.$$
When $\lambda=0$, by elliptic regularity theory a nonnegative $p-$harmonic function $G_0$ in $\Omega \setminus \{x_0\}$ belongs to $C_{\text{loc}}^{1,\alpha}(\Omega \setminus \{x_0\})$ for some $\alpha \in (0,1)$ and, according to \cite{serrin}, behaves - if singular -  like the fundamental solution
$$\Gamma(x)= \left\{\begin{array}{ll}
\frac{C_0}{|x-x_0|^\frac{N-p}{p-1}} &\hbox{if }1<p<N\\
- (N\omega_N)^{-\frac{1}{N-1}} \log |x-x_0|&\hbox{if }p=N \end{array} \right.$$
of $ -\Delta_p \Gamma=\delta_{x_0}$ in $\R^N$, where $C_0=\frac{p-1}{N-p}(N\omega_N)^{-\frac{1}{p-1}}$ and $\omega_N$ is the measure of the unit ball in $\R^N$. By a combination of scaling arguments and regularity estimates,  Kichenassamy and Veron \cite{kichenassamyveron} showed that, in the singular situation, up to a re-normalization, $G_0$ is a solution of
\begin{equation} \label{1010}
-\Delta_p G = \delta_{x_0} \qquad \text{in } \Omega
\end{equation}
and differs from $\Gamma$ by a locally bounded function $H_0=G_0-\Gamma$ in $\Omega$. Given $g \in L^\infty(\Omega) \cap W^{1,p}(\Omega)$, a solution $G_0 \in W_{\text{loc}}^{1,p}(\Omega \setminus \{x_0\}) \cap W^{1,p-1}(\Omega)$
to \eqref{1010} with $G_0 \Big|_{\partial \Omega}=g$ can be found in many different ways (see for example  \cite{kichenassamyveron,serrin}) and turns out to be unique thanks to the property $\nabla H_0=o(|\nabla \Gamma|)$ as $x\to x_0$. As noticed in \cite{kichenassamyveron}, the same approach via scaling arguments leads to a continuity property of $H_0$ at $x_0$.

\medskip \noindent The aim of the present paper is to  establish the H\"older continuity of $H_\lambda=G_\lambda-\Gamma$ at $x_0$ when $\lambda=0$ and to include the case $\lambda <\lambda_1$. Notice that such H\"older property is new already when $\lambda=0$ and is relevant since Green's functions naturally arise in the description of concentration phenomena for quasi-linear PDE's, see for example \cite{AnEs2}, even if representation formulas are no-longer available in a quasi-linear context. Since the seminal works \cite{LaUr,serrin, serrin65} in the sixties, the regularity theory for quasi-linear elliptic problems has been first refined in \cite{Eva,lewis} in the $p-$harmonic setting, see also \cite{Ura}, and then in \cite{dib,lieberman,tolksdorf} for general $p-$Laplace type equations.  To treat the case of a Radon measure as right hand side,  a general existence and uniqueness theory has been developed,  both in the scalar and vectorial case,  through different approaches: renormalized solutions, see for instance \cite{DMOP,Mur}; entropy solutions or SOLA (solutions obtained as limit of approximations) in \cite{AgPe,BBGGPV,boccardogallouet1,boccardogallouet2}; in weak Lebesgue spaces \cite{DHM1,DHM2,DHM3}; in grand Sobolev spaces \cite{GIS}. A powerful and general approach has also been developed through a potential theory in nonlinear form, see for example \cite{HKM,KuMi} for an overview on old and recent achievements.  Also in the simplest case $\lambda=0$ the problem we are interested in does not fit into these general theories and a different approach, based on a new but rather simple idea, is necessary. The main point is to consider $H_\lambda$ as a solution of
\begin{equation} \label{1123}
-\Delta_p (\Gamma+H_\lambda)+\Delta_p \Gamma= \lambda G_\lambda^{p-1}  \qquad \text{in } \Omega \setminus \{x_0 \}
\end{equation}
for any $G_\lambda=\Gamma+H_\lambda$ solving \eqref{Gproblem_phdthesis} below and to apply the Moser iterative scheme in \cite{serrin} to derive H\"older estimates on $H_\lambda$ thanks to the coercivity of the difference operator,  as expressed by the estimate
\begin{equation} \label{aboveestimate}
\inf_{X\not= Y}  \frac{\langle |X+Y|^{p-2} (X+Y) - |X|^{p-2} X , Y \rangle}{(|X| + |Y|)^{p-2} |Y|^2}>0.
\end{equation}
When $p\geq 2$ gradient $L^p-$estimates on $H_\lambda$ can be derived for the difference equation \eqref{1123} as in the pure $p-$Laplace case and the only difficulty, when performing local estimates, comes from the failure of good upper estimates on $|\nabla \Gamma+\nabla H_\lambda|^{p-2}(\nabla \Gamma+\nabla H_\lambda)-|\nabla \Gamma|^{p-2} \nabla \Gamma$, caused by the singular behavior of $\nabla \Gamma$ at $x_0$.  Since the inequality $(|X|+|Y|)^{p-2}|Y|^2 \geq \delta |Y|^p$, $\delta>0$,  is no longer true for $1<p<2$, one realizes that the difference equation \eqref{1123} differs from the pure $p-$Laplace case and weighted gradient $L^2-$estimates on $H_\lambda$ are the natural ones one can hope for.

\medskip \noindent Let us first discuss the case $\lambda=0$, which is the most relevant since it concerns the behavior of $p-$harmonic functions at isolated singularities.  In the two-dimensional situation a very precise description has been provided in \cite{manfredi}, whereas for $N\geq 2$ the only available result concerns the continuity of $H_0$ and has been given in \cite{kichenassamyveron}, as already discussed.  A special attention is paid here to avoid any restrictions on  $p$ and our first main result below improves in full generality what was previously known:
\begin{theorem} \label{mainth0}
Let $\Omega \subset \R^N$ be a bounded domain, $x_0 \in \Omega$ and $1<p\leq N$.  The unique nonnegative solution $G_0$ to
$$\begin{cases}
-\Delta_p G = \delta_{x_0} \qquad &\text{in } \Omega\\
G=0 &\text{on } \partial \Omega \end{cases}$$
satisfies 
\begin{equation} \label{122bis}
\nabla(G_0-\Gamma) \in L^{\bar q}(\Omega),\quad \bar q =\frac{N(p-1)}{N-1},
\end{equation}
and the regular part $H_0=G_0-\Gamma$ is H\"older continuous at $x_0$.
\end{theorem}
Let us stress that the integrability condition \eqref{122bis} can be improved into $\nabla H_0 \in L^p(\Omega)$ if $p\geq 2$.  Since $\nabla \Gamma \in L^q(\Omega)$ for all $q<\bar q$, the exponent $\bar q$ represents the threshold gradient-integrability which distinguishes the singular situation from the non-singular one and the property \eqref{122bis} is crucial, when running the Moser iterative scheme, to use appropriate test functions $\Psi(H_\lambda)$ into \eqref{1123} as the equation were valid in the whole $\Omega$.  The validity of higher regularity properties for $H_0$ represents a challenging open question in this context. 

\medskip \noindent Let us now address the case $\lambda \not=0$ and consider the problem
\begin{equation} \label{Gproblem_phdthesis}
\begin{cases}
-\Delta_p G - \lambda G^{p-1}= \delta_{x_0} \qquad &\text{in } \Omega\\
G\geq0 &\text{in } \Omega\\
G=0 &\text{on } \partial \Omega.
\end{cases}
\end{equation}
Our second main result is the following:
\begin{theorem} \label{mainth}
Let $\Omega \subset \R^N$ be a bounded domain, $x_0 \in \Omega$ and $2\leq p\leq N$.  If $\lambda<\lambda_1$ with $\lambda \not=0$,  problem \eqref{Gproblem_phdthesis} has a solution $G_\lambda$ with 
\begin{equation} \label{122}
\nabla(G_\lambda-\Gamma) \in L^{\bar q}(\Omega),\quad \bar q =\frac{N(p-1)}{N-1},
\end{equation}
which is unique in the class of solutions satisfying \eqref{122}. Moreover, the regular part $H_\lambda=G_\lambda-\Gamma$ is H\"older continuous at $x_0$ if $p>\frac{N}{2}$.
\end{theorem}
\noindent  Some comments are in order.  While \eqref{122bis} is proved to be true for $G_0$, for $\lambda \not=0$ we cannot guarantee the validity of \eqref{122} for any solution $G_\lambda$. However, since \eqref{122} is generally valid for all solutions obtained through an approximation scheme, assumption \eqref{122} in Theorem \ref{mainth} is a rather natural request which - at the same time - allows us to show uniqueness of $G_\lambda$ when $p\geq 2$ and H\"older continuity of $H_\lambda$ when $p>\frac{N}{2}$.  In view of  $H_\lambda \in L^\infty(\Omega)$ and
$$\Gamma \in L^q(\Omega)\quad \hbox{for }1 \leq q<\bar q^*, \: \bar q^*= \left\{ \begin{array}{ll} \frac{N(p-1)}{N-p}&\hbox{if }1<p<N\\ +\infty &\hbox{if }p=N, \end{array}\right.$$
notice that condition $p>\frac{N}{2}$ ensures $G_\lambda^{p-1} \in L^q(\Omega)$ for some $q>\frac{N}{p}$ in  \eqref{1123},  a natural condition arising in \cite{serrin}  to prove $L^\infty-$bounds.  In this respect,  observe that also in the semilinear case $p=2$ the function $H_\lambda$ is no longer regular at $x_0$ when $2=p \leq \frac{N}{2}$. 

\medskip \noindent The paper is organized as follows. Section $2$ is devoted to establish the existence part in Theorems \ref{mainth0} and \ref{mainth} along with some $L^\infty-$estimates, while uniqueness issues are addressed in Section $3$. Harnack inequalities and H\"older estimates for $H_\lambda$ are established in Section $4$. For easy of notations, we will just consider the case $x_0=0$.\\ 
The results of the present paper are crucial in \cite{AnEs2} to discuss existence results for a quasi-linear elliptic equation of critical Sobolev growth \cite{brezisnirenberg,gueddaveron2} in the low-dimensional case as in \cite{Dru,Esp}.

\section{Existence of Green's functions}
Given $g \in L^\infty(\Omega) \cap W^{1,p}(\Omega)$, set $W^{1,q}_g(\Omega)=g+W^{1,q}_0(\Omega)$ for all $q \geq 1$ and consider 
$$\lambda_{1,g}= \inf_{u \in W_g^{1,p}(\Omega) \setminus \{0\}}  \frac{\int_\Omega |\nabla u|^p}{\int_\Omega |u|^p}.$$
Since the minimizer $\tilde g$ of $\int_\Omega |\nabla u|^p $ in $W^{1,p}_g(\Omega)$ is a $p-$harmonic function in $\Omega$ so that $\|\tilde g\|_\infty \leq \|g\|_\infty$,  we assume that either $g=0$ or $g \in L^\infty(\Omega) \cap W^{1,p}(\Omega)$ is a $p-$harmonic and non-constant function in $\Omega$ so to guarantee $\lambda_{1,g}>0$.

\medskip \noindent For $g\geq 0$ and $\lambda<\lambda_{1,g}$ let us discuss the problem
\begin{equation} \label{353}
\begin{cases}
-\Delta_p G - \lambda G^{p-1}= \delta_0 \qquad &\text{in } \Omega\\
G\geq0 &\text{in } \Omega\\
G=g &\text{on } \partial \Omega
\end{cases}
\end{equation}
with 
\begin{equation} \label{1440}
g \in L^\infty(\Omega) \cap W^{1,p}(\Omega) \hbox{ $p-$harmonic in }\Omega, \ g \hbox{ non-constant unless }g=0.
\end{equation}
Solutions of \eqref{353} are found by an approximation procedure based either on removing small balls $B_\epsilon(0)$ when $\lambda=0$ as in \cite{kichenassamyveron} or on approximating $\delta_0$ by smooth functions when $\lambda \not=0$ as in \cite{AgPe,BBGGPV,boccardogallouet1,boccardogallouet2}. We have the following existence result.
\begin{theorem} \label{theoremexistenceG}
Let $1<p\leq N$, $g \geq 0$ satisfying \eqref{1440}, $\lambda<\lambda_{1,g}$ and assume $p\geq 2$ only when $\lambda \not=0$. Then there exists a solution $G_\lambda$ of problem \eqref{353} so that $H_\lambda=G_\lambda-\Gamma $ satisfies \eqref{122}. Moroever,  there holds $H_\lambda \in L^\infty(\Omega)$ whenever either $\lambda =0$ or $\lambda \not=0$, $p>\frac{N}{2}$.
\end{theorem}
\begin{proof} Consider first the case $\lambda=0$. We repeat the argument in \cite{kichenassamyveron} and the only point is to establish suitable bounds on $H_0=G_0-\Gamma $. Let $G_\epsilon$ be the $p-$harmonic function  in $\Omega_\epsilon=\Omega \setminus B_\epsilon(0)$ so that $G_\epsilon=g$ on $\partial \Omega$ and $G_\epsilon =\Gamma$ on $\partial B_\epsilon(0)$. Since $\Gamma$ is a positive $p-$harmonic function in $\Omega \setminus \{0\}$, by comparison principle we deduce that $G_\epsilon \geq 0$ and $|G_\epsilon-\Gamma | \leq C_0$ in $\Omega_\epsilon$, with $C_0=\|g\|_\infty+\|\Gamma\|_{\infty,\partial \Omega}$. By elliptic estimates  \cite{Eva,lewis,Ura} for $p-$harmonic functions we deduce that $G_\epsilon$ is uniformly bounded in $C^{1,\alpha}_{\hbox{loc}} (\Omega \setminus \{0\})$. By Ascoli-Arzel\'a Theorem we can find a sequence $\epsilon_n \to 0$ so that $G_n:=G_{\epsilon_n} \to G_0$ in $C^1_{\hbox{loc}} (\Omega \setminus \{0\})$ as $n \to +\infty$, where $G_0 \geq 0$ is a $p-$harmonic function in $\Omega \setminus \{0\}$ so that 
\begin{equation} \label{1714}
H_0=G_0-\Gamma \in L^\infty (\Omega). 
\end{equation}
Letting $\eta$ be a cut-off function with $\eta=1$ near $\partial \Omega$ and $\eta=0$ near $0$, use $\eta^p (G_\epsilon-g) \in W^{1,p}_0(\Omega_\epsilon)$ as a test function for $\Delta_p G_\epsilon=0$ in $\Omega_\epsilon$ to get
\begin{equation} \label{1658}
\int_{\Omega_\epsilon} \eta^p \langle |\nabla G_\epsilon|^{p-2} \nabla G_\epsilon,  \nabla (G_\epsilon-g) \rangle=- p \int_{\Omega_\epsilon} \eta^{p-1} (G_\epsilon-g)  \langle |\nabla G_\epsilon|^{p-2} \nabla G_\epsilon,   \nabla \eta \rangle \leq C
\end{equation}
in view of $\nabla \eta=0$ near $\partial \Omega $ and $0$. Since
$$\int_{\Omega_\epsilon} \eta^p |\nabla G_\epsilon|^{p-1} | \nabla g|\leq 
\frac{1}{2} \int_{\Omega_\epsilon} \eta^p |\nabla G_\epsilon|^p+C \int_{\Omega_\epsilon} \eta^p |\nabla g|^p$$
for some $C>0$ in view of the Young inequality, by \eqref{1658} we deduce that $G_\epsilon$ is uniformly bounded in $W^{1,p}$ near $\partial \Omega$. Then $G_0=g$ on $\partial \Omega$ and $G_0$ solves \eqref{353} with $ \lambda=0$ in view of \eqref{1714} and \cite{kichenassamyveron,serrin}. 

\medskip \noindent Moreover, use $(1-\eta)(G_\epsilon-\Gamma) \in W^{1,p}_0(\Omega_\epsilon)$ as a test function for $-\Delta_p G_\epsilon+\Delta_p \Gamma=0$ in $\Omega_\epsilon$ to get
\begin{equation} \label{1411}
\int_{\Omega_\epsilon} (1-\eta) \langle |\nabla G_\epsilon|^{p-2} \nabla G_\epsilon- |\nabla \Gamma|^{p-2} \nabla \Gamma, \nabla (G_\epsilon-\Gamma) \rangle \leq C
\end{equation}
in view of $\nabla \eta=0$ near $\partial \Omega $ and $0$. By the coercivity estimate \eqref{aboveestimate} and the uniform $W^{1,p}-$bound on $G_\epsilon$ and $\Gamma$ away from $0$ we deduce that \eqref{1411} implies 
\begin{equation} \label{1542}
\int_{\Omega_\epsilon} (|\nabla \Gamma|+|\nabla H_\epsilon|)^{p-2}|\nabla H_\epsilon|^2\leq C
\end{equation}
for some uniform constant $C>0$, where $H_\epsilon=G_\epsilon-\Gamma$. When $p \geq 2$ estimate \eqref{1542} implies 
$$\nabla H_0 \in L^p(\Omega)$$
thanks to the Fatou convergence Theorem along the sequence $\epsilon_n$. For $1<p< 2$ by \eqref{1542} and the H\"older inequality we get
\begin{eqnarray*}
\int_{\Omega_\epsilon} |\nabla H_\epsilon|^{\bar{q}} &=& \int_{\Omega_\epsilon} (|\nabla \Gamma| + |\nabla H_\epsilon|)^\frac{(p-2)\bar{q}}{2} |\nabla H_\epsilon|^{\bar{q}}  (|\nabla  \Gamma|+ |\nabla H_\epsilon|)^\frac{(2-p)\bar{q}}{2} \leq C (\| \nabla \Gamma \|_ {s,\Omega_\epsilon}^{\frac{(2-p)\bar q}{2}}  +\|\nabla H_\epsilon\|_ {s,\Omega_\epsilon}^{\frac{(2-p)\bar q}{2}})\\
&\leq& C (\| \nabla \Gamma \|_ {s,\Omega_\epsilon}^{\frac{(2-p)\bar q}{2}}  +|\Omega_\epsilon|^{\frac{(2-p)(\bar q -s)}{ 2s}} \|\nabla H_\epsilon\|_ {\bar q,\Omega_\epsilon}^{\frac{(2-p)\bar q}{2}})
\end{eqnarray*}
for some $C>0$ and $s=\frac{N(p-1)(2-p)}{3N-2-Np}$,  thanks to $s<\bar q$ in view of $p<2\leq N$.  By $\nabla \Gamma \in L^q(\Omega)$ for all $q<\bar q$ and the Young inequality we finally obtain $\int_{\Omega_\epsilon} |\nabla H_\epsilon|^{\bar{q}} \leq C$ for some uniform constant $C>0$ and then
$$\nabla H_0 \in L^{\bar q}(\Omega)$$
does hold in the case $1<p<2$ thanks to the Fatou convergence Theorem.

\medskip \noindent  Once the case $\lambda=0$ has been treated,  assume $p\geq 2$ and follow the approach in \cite{AgPe,BBGGPV,boccardogallouet1,boccardogallouet2}. Notice that for $\lambda=0$ we provide below an efficient approximation scheme which is different from the previous one. Consider a sequence $0\leq f_n \in C_0^\infty(\Omega)$ so that $f_n \rightharpoonup \delta_0$ weakly in the sense of measures in $\Omega$ with $\sup_n \|f_n\|_1<+\infty$ and $f_n \to 0$ locally uniformly in $\Omega \setminus \{0\}$ as $n \to + \infty$. Since $\lambda<\lambda_{1,g}$ and $g,f_n \geq 0$, the minimization of
\begin{equation*}
\frac{1}{p} \int_\Omega |\nabla u|^p - \frac{\lambda}{p} \int_\Omega |u|^p - \int_\Omega f_n u, \quad u \in W_g^{1,p}(\Omega),
\end{equation*}
provides a nonnegative solution  $G_n \in W_g^{1,p}(\Omega)$ to
\begin{equation} \label{Gjproblem_boundedness_Lpnorm}
-\Delta_p G_n - \lambda G_n^{p-1}= f_n \quad \text{ in } \Omega.
\end{equation}
We use here Lemmas \ref{lemmaGj} and \ref{lemmaaggiunto} below to show first that $G_n^{p-1}$ is uniformly bounded in $L^1(\Omega)$ and then, up to a subsequence, $G_n \to G_\lambda$ in $W_g^{1,q}(\Omega)$ as $n \to +\infty$ for some $G_\lambda$ and for all $1\leq q<\bar{q}$.  By the Sobolev embedding Theorem we have that $G_n \to G_\lambda$ in $L^q(\Omega)$ as $n \to +\infty$ for all $1\leq q<\bar q^*$ and in particular in $L^{p-1}(\Omega)$. Therefore one can pass to the limit in \eqref{Gjproblem_boundedness_Lpnorm} and get that $G_\lambda\geq 0$ solves \eqref{353} in view of $\bar q>p-1$.

\medskip \noindent In order to establish suitable bounds on $H_\lambda=G_\lambda-\Gamma $, let $0\leq \tilde G_n \in W^{1,p}_g(\Omega)$ be the solution of 
$$- \Delta_p \tilde G_n = f_n \quad \hbox{ in }\Omega,$$
obtained as a minimizer of $\frac{1}{p} \int_\Omega |\nabla u|^p - \int_\Omega f_n u$ in $W^{1,p}_g(\Omega)$ in view of $\lambda_{1,g}>0$. Arguing as for \eqref{Gjproblem_boundedness_Lpnorm}, we deduce that, up to a subsequence, $\tilde G_n \to \tilde G$ in $W^{1,q}_g(\Omega)$ as $n \to +\infty$ for all $1\leq q< \bar{q}$, where $\tilde G \geq 0$ solves $- \Delta_p \tilde G = \delta_0$ in $\Omega$. By \cite{serrin} and the uniqueness result in \cite{kichenassamyveron} we have that $\tilde G=G_0$ and $\tilde H=\tilde G-\Gamma=H_0$. 
Since $-\Delta_p G_n +\Delta_p \tilde G_n=\lambda G_n^{p-1}$ in $\Omega$ with $G_n=\tilde G_n$ on $\partial \Omega$, by Lemma \ref{lemmaaggiunto} we deduce that $\displaystyle \sup_n\| \nabla (G_n-\tilde G_n)\|_{\bar q}<+\infty$  in view of $\displaystyle \sup_n \|G_n^{p-1}\|_m< +\infty$ for all $1\leq m<\frac{\bar q^*}{p-1}$. Since $\nabla(G_n-\tilde G_n) \to \nabla(H_\lambda-H_0)$ a.e. in $\Omega$ as $n \to +\infty$ and $\nabla H_0$ satisfies \eqref{122bis}, by the Fatou convergence Theorem we obtain that $\nabla H_\lambda$ satisfies \eqref{122}.  If either $\lambda=0$ or $\lambda \not=0$, $p>\frac{N}{2}$ a $L^\infty-$bound on $H_\lambda$ follows by Theorem \ref{Hbounded} below and the proof is complete.
\end{proof}
\noindent The following result has been crucially used in the proof of Theorem \ref{theoremexistenceG} and in its proof we closely follow a tricky idea in \cite{orsina} combined with some apriori estimates given in Lemma \ref{lemmaaggiunto} below.
\begin{lemma} \label{lemmaGj} Let $2\leq p\leq N$. Assume that $a_n \in L^\infty(\Omega)$, $f_n \in L^1(\Omega)$, $g_n$ satisfy \eqref{1440} and
\begin{equation} \label{73947}
\lim_{n \to +\infty} \|a_n -a\|_\infty= 0, \quad \sup_\Omega a <\lambda_1, \quad \sup_{n \in \mathbb N} \left[\|f_n\|_1 +\|g_n\|_\infty \right]<+\infty.
\end{equation}
If $u_n \in W^{1,p}_{g_n}(\Omega)$ is a sequence of solutions to
$$-\Delta_p u_n - a_n |u_n|^{p-2}u_n= f_n \quad \text{ in } \Omega,$$
then $\displaystyle \sup_{n \in \mathbb N} \|u_n\|_{p-1}<+\infty$.
\end{lemma} 
\begin{proof} Assume by contradiction that
\begin{equation} \label{1838}
\|u_n\|_{p-1} \to + \infty\qquad \hbox{as } n \to +\infty.
\end{equation}
Setting $\hat{u}_n = \frac{u_n}{\|u_n\|_{p-1}}$, $\hat{f}_n = \frac{f_n}{\|u_n\|_{p-1}^{p-1}}$ and $\hat{g}_n = \frac{g_n}{\|u_n\|_{p-1}}$, we have that $\hat{u}_n$ solves
\begin{equation} \label{hatGjequation_boundedness_Lpnorm}
\begin{cases}
- \Delta_p \hat{u}_n - a_n |\hat{u}_n|^{p-2}\hat{u}_n = \hat{f}_n  \qquad&\text{in } \Omega\\
\hat{u}_n=\hat{g}_n &\text{on } \partial \Omega
\end{cases}
\end{equation}
with 
\begin{equation} \label{01841}
\|\hat{u}_n\|_{p-1}=1, \ \sup_{n \in \mathbb N}\|a_n\|_\infty <\infty,\ \|\hat {f}_n \|_{L^1(\Omega)} +\| \hat{g}_n\|_\infty \to 0 \hbox{ as }n \to +\infty
\end{equation}
in view of \eqref{73947}-\eqref{1838}. Fix $p-1< p_0<\bar q$ and define $p_j=\frac{N^2(p-1) p_{j-1}}{(N+1)[N(p-1)- p_{j-1}]}$ in a recursive way for $j \geq 1$.  Notice that $\frac{N(p-1)}{N+1} < p_j<p_{j+1}$ by induction and there exists a unique $J\geq 0$ so that $p_0,\ldots,p_{J-1} \leq \frac{Np(p-1)}{Np-N+p} < p_J$.  Since $\Delta_p \hat g_n=0$ in $\Omega$, by Lemma \ref{lemmaaggiunto} with $m=1$ we get that $\hat u_n-\hat g_n $ is uniformly bounded in $W^{1,q}_0(\Omega)$ for all $1\leq q<\bar q$ in view of \eqref{hatGjequation_boundedness_Lpnorm}-\eqref{01841} and then, up to a subsequence,  $\hat u_n -\hat g_n \rightharpoonup v^0$  in $W^{1,p_0}(\Omega)$ as $n \to +\infty$. Define $ v_n^0=\hat u_n$ and $v_n^j \in W^{1,p}_{\hat g_n}$ as the solution of $-\Delta_p v_n^j=a_n |v_n^{j-1}|^{p-2}v_n^{j-1}$ in $\Omega$ in view of $\lambda_{1,\hat g_n}=\lambda_{1,g_n}>0$.  
Lemma \ref{lemmaaggiunto}, applied to $v_n^1-\hat g_n$ with $m=\frac{p_{0}}{p-1}\leq \frac{Np}{Np-N+p}$, $q=\frac{N}{N+1} \frac{mN(p-1)}{N-m}$ and to $v_n^1-v_n^0$ with $m=1$, $q=p_0$ in view of \eqref{hatGjequation_boundedness_Lpnorm}-\eqref{01841},  provides that, up to a subsequence, $v_n^1-\hat g_n  \rightharpoonup v^1$ in $W^{1,p_1}_0(\Omega)$ and $v_n^1-v_n^0 \to 0 $  in $W^{1,p_0}_0(\Omega)$ as $n \to +\infty$.  By iterating we deduce that, up to a subsequence, $v_n^j-\hat g_n  \rightharpoonup v^j$ in $W^{1,p_j}_0(\Omega)$ and $v_n^j-v_n^{j-1} \to 0$ in $W^{1,p_{j-1}}_0(\Omega)$ as $n \to +\infty$ for all $j=1,\dots,J$.  Since $a_n |v_n^J|^{p-2}v_n^J$ is uniformly bounded in $L^m(\Omega)$ with $m=\frac{p_J}{p-1} > \frac{Np}{Np-N+p}$,  by Lemma \ref{lemmaaggiunto} we deduce that, up to a subsequence,  $v_n^{J+1}-\hat g_n\rightharpoonup v^{J+1}$ in $W^{1,p}_0(\Omega)$ as $n \to +\infty$.  At the same time,  by Lemma  \ref{lemmaaggiunto} $v_n^{J+1}-v_n^J \to 0$ in $W^{1,p_{J}}_0(\Omega)$ as $n \to +\infty$. Since $v_n^j-v_n^{j-1} \to 0$ in $W^{1,p_0}_0(\Omega)$ and  $v_n^j-v_n^{j-1}=(v_n^j-\hat g_n)-(v_n^{j-1}-\hat g_n) \rightharpoonup v^j-v^{j-1}$ weakly in $W^{1,p_0}_0(\Omega)$ as $n \to +\infty$ for all $j=1,\dots,J+1$,  we deduce that $v^0=\ldots=v^{J+1}$ and then $\hat u_n -\hat g_n\rightharpoonup v^0$ in $W_0^{1,p_0}(\Omega)$ as $n \to +\infty$ with $v^0=v^{J+1} \in W^{1,p}_0(\Omega)$. 

\medskip \noindent  Let us compare $\hat u_n$ with $z_n \in W^{1,p}_0(\Omega)$, solution to 
\begin{equation} \label{1529}
-\Delta_p z_n= a_n |\hat{u}_n|^{p-2}\hat{u}_n +\hat{f}_n  \quad \hbox{ in }\Omega.
\end{equation}
Since $|\hat u_n-z_n| \leq \|\hat g_n\|_\infty$ on $\partial \Omega$, by the weak maximum principle we deduce that $\|\hat u_n-z_n\|_\infty \leq \|\hat g_n\|_\infty$.  By \eqref{01841}-\eqref{1529} and Lemma \ref{lemmaaggiunto} we deduce that, up to a subsequence and for some $z^0$, there holds
\begin{equation} \label{1549}
z_n \to z^0 \quad \hbox{ in }W^{1,q}_0(\Omega), \ 1\leq q<\bar q.
\end{equation}
By testing $-\Delta_p \hat u_n+\Delta_p z_n=0$ in $\Omega$ against $\eta^p (\hat u_n-z_n)$, $0\leq \eta \in C_0^\infty(\Omega)$,  one gets 
\begin{eqnarray*}
\int_\Omega \eta^p |\nabla (\hat u_n-z_n)|^p &\leq& C' \int_\Omega \eta^{p-1} |\nabla \eta| (|\nabla (\hat u_n-z_n)|^{p-2}+|\nabla z_n|^{p-2}) |\nabla (\hat u_n-z_n)|
|\hat u_n-z_n|\\
&\leq& \frac{1}{2} \int_\Omega \eta^p |\nabla (\hat u_n-z_n)|^p+C \left( \|\hat g_n\|_\infty^p+ \|\hat g_n\|_\infty^{\frac{p}{p-1}} \|\nabla z_n\|^{p-2}_{\frac{p(p-2)}{p-1}}\right) \to 0
\end{eqnarray*}
as $n \to +\infty$ in view of the Young's inequality and \eqref{01841}.  We have used that $\displaystyle \sup_n \|\nabla z_n\|_{\frac{p(p-2)}{p-1}}<+\infty$ thanks to \eqref{1549} and $\frac{p(p-2)}{p-1}<\bar q$.  Since $\nabla (\hat u_n -z_n) \to 0$ locally in $L^p-$norm as $n \to +\infty$,  by \eqref{1549} we deduce that
\begin{equation} \label{1555}
\hat u_n  \to v^0   \quad \hbox{ in }L^{p-1}(\Omega) \hbox{ and }W^{1,q}(\Omega'),\quad \forall \ \Omega' \subset \subset \Omega, \ \ \forall \ 1\leq q< \bar q,
\end{equation}
in view of $\|\hat u_n-z_n\|_\infty \leq   \|\hat g_n\|_\infty \to 0$ and $\hat u_n -\hat g_n\rightharpoonup v^0$ in $W_0^{1,p_0}(\Omega)$ as $n \to +\infty$ for $p_0 \geq p-1$.

\medskip \noindent By \eqref{hatGjequation_boundedness_Lpnorm} and \eqref{1555} we have that $v^0 \in W^{1,p}_0(\Omega)$ solves
\begin{equation} \label{hatGequation_boundedness_Lpnorm}
-\Delta_p v^0 - a |v^0|^{p-2} v^0= 0 \qquad\text{in } \Omega
\end{equation}
in view of \eqref{73947} and \eqref{01841}.  Since 
$$\int_\Omega |\nabla v^0|^p -  \int_\Omega a |v^0|^p= 0 $$
by integration of \eqref{hatGequation_boundedness_Lpnorm} against $v^0 \in W^{1,p}_0(\Omega)$, by $\displaystyle \sup_\Omega a<\lambda_1$ one finally deduces that $v^0=0$ and then $\hat{u}_n \to 0$ in $L^{p-1}(\Omega)$, in contradiction with $\|\hat{u}_n\|_{p-1}=1$.
\end{proof}
The results in \cite{AgPe,boccardogallouet1,boccardogallouet2},  valid for homogeneous boundary values, can be easily extended to non-homogeneous ones when $p\geq 2$,  as discussed for instance in the Appendix of \cite{AgPe} when $p=N$. For the sake of completeness, we reproduce it here in the following simplest form,  sufficient for our purposes:
\begin{lemma} \label{lemmaaggiunto} 
Let $2\leq p\leq N$.  Assume $\|f_1- f_2\|_m \leq C_0$ for some $C_0>0$ and either $1\leq m \leq \frac{Np}{Np-N+p}$, $1\leq q < \frac{m N(p-1)}{N-m}$ or $m > \frac{Np}{Np-N+p}$, $1\leq q \leq p$. Then there exists $C>0$ so that $\|\nabla (u_1-u_2)\|_q \leq C \|f_1-f_2\|_m^{\frac{1}{p}}$ for all solutions $u_1,u_2 \in W^{1,p}(\Omega)$ of  $-\Delta_p u_i=f_i$, $i=1,2$,  in $\Omega$ with $u_1=u_2$ on $\partial \Omega$. 
\\
 Moreover, given $g$ satisying \eqref{1440} the set of solutions $u\in W_g^{1,p}(\Omega)$ of $-\Delta_p u=f$ in $\Omega$ with $\|f\|_1  \leq C_0$ is relatively compact in $W^{1,q}(\Omega)$ for all $1\leq q < \bar q$.
\end{lemma}
\begin{proof} 
Let  $u_1,u_2 \in W^{1,p}(\Omega)$ be solutions of $-\Delta_p u_i=f_i$,  $i=1,2$,  in $\Omega$ with $u_1=u_2$ on $\partial \Omega$.  Take $T_{k,l}$, $0\leq k \leq l$, as the odd function so that
\begin{equation}\label{trunc}
T_{k,l}(s)= \min \{ \max\{s-k, 0\},l-k \} \quad \hbox{ in }[0,+\infty)
\end{equation} 
and use $T_{k,k+1}(u_1-u_2)$ as a test function to get
$$ \int_{\{k\leq |u_1-u_2| < k+1 \} } \langle |\nabla u_1|^{p-2} \nabla u_1-|\nabla u_2|^{p-2} \nabla u_2, \nabla (u_1-u_2) \rangle=\int_\Omega (f_1-f_2) T_{k,k+1}(u_1-u_2),$$
which implies
\begin{equation} \label{1846}
 \int_{\{k\leq |u_1-u_2| < k+1 \} }  |\nabla (u_1-u_2)|^p \leq C \|f_1-f_2\|_m |\{ |u_1-u_2|\geq k \} |^{\frac{m-1}{m}}
\end{equation} 
in view of \eqref{aboveestimate} and $p\geq 2$.  By \eqref{1846} the function $v=u_1-u_2 \in W^{1,p}_0(\Omega)$ satisfies 
\begin{equation} \label{9856}
\int_{B_k} |\nabla v|^p \leq c_0 |E_k|^{\frac{m-1}{m}}, \ k \geq 0,
\end{equation}
with $c_0=C\|f_1-f_2\|_m$,  where $E_k=\{|v| \geq k\}$ and $B_k=E_k \setminus E_{k+1}$.

\medskip \noindent Consider first the case $1\leq m \leq \frac{Np}{Np-N+p}$,  $1\leq q < \frac{m N(p-1)}{N-m}$ and set $q^*=\frac{Nq}{N-q}$.  Since $q<\frac{m N(p-1)}{N-m}\leq p$ thanks to $m\leq \frac{N p}{N p-N+p}$ and 
\begin{equation} \label{910125}
\int_{B_k} |\nabla v|^q \leq (\int_{B_k} |\nabla v|^p)^{\frac{q}{p}}  |B_k|^{\frac{p-q}{p}}
\end{equation}
in view of the H\"older inequality,  by \eqref{9856} we obtain that
\begin{eqnarray*}
\int_{B_k} |\nabla v|^q \leq  c_0^{\frac{q}{p}}  \| v\|_{q^*}^{\frac{qq^*(m-1)}{p m}}  ( \int_{B_k} |v|^{q^*} )^{\frac{p-q}{p}}
\frac{1}{k^{\frac{q^*(p m-q)}{p m}}}
\end{eqnarray*}
for all $k\geq 1$ thanks to 
$$|B_k| \leq k^{- q^*} \int_{B_k} |v|^{q^*}, \quad |E_k| \leq k^{-q^*} \int_{\Omega} |v|^{q^*}.$$
Summing up and still by H\"older's inequality one deduces
\begin{eqnarray*} 
\int_{\{|v|\geq k_0\}} |\nabla v|^q \leq  c_0^{\frac{q}{p}}  \| v\|_{q^*}^{\frac{qq^*(m-1)}{p m}}  
(\sum_{k=k_0}^\infty \int_{B_k} |v|^{q^*} )^{\frac{p-q}{p}}(\sum_{k=k_0}^\infty \frac{1}{k^{\frac{q^*(p m-q)}{m q}}})^{\frac{q}{p}}
\end{eqnarray*}
and then
\begin{eqnarray} \label{91039} 
\int_\Omega |\nabla v|^q  \leq  k_0 c_0^{\frac{q}{p}} |\Omega|^{\frac{p m-q}{p m}}+c_0^{\frac{q}{p}}  \| v\|_{q^*}^{ \frac{q^*(p m-q)}{p m}}
(\sum_{k=k_0}^\infty \frac{1}{k^{\frac{q^*(p m-q)}{m q}}})^{\frac{q}{p}} 
\end{eqnarray}
for  a given $k_0\in \mathbb{N}$ in view of \eqref{9856}-\eqref{910125} for $k=0,\ldots, k_0-1$.  Since $\frac{q^*(p m-q)}{p m}\leq q$,  by Young's inequality \eqref{91039} implies in turn that
\begin{eqnarray} \label{91039bis} 
\int_\Omega |\nabla v|^q  \leq  k_0 c_0^{\frac{q}{p}} |\Omega|^{\frac{p m-q}{p m}}+C c_0^{\frac{q}{p}}  (\| v\|_{q^*}^q+1)
(\sum_{k=k_0}^\infty \frac{1}{k^{\frac{q^*(p m-q)}{m q}}})^{\frac{q}{p}}.
\end{eqnarray}
Since $\frac{q^*(p m-q)}{m q}>1$ thanks to $q <\frac{m N(p-1)}{N-m}$, the series in \eqref{91039bis} is convergent and we can choose $k_0$ sufficienty large (depending on $C_0$) so that $\|v\|_{q^*} \leq C' c_0^{\frac{1}{p}}$ and then $\| \nabla v \|_q \leq C c_0^{\frac{1}{p}}$ in view of the Sobolev embedding Theorem, where the last estimate gets rewritten as 
\begin{equation} \label{1005}
\|\nabla (u_1-u_2)\|_q \leq C \|f_1-f_2\|_m^{\frac{1}{p}} .
\end{equation}
Consider now the case $m > \frac{Np}{Np-N+p}$, $1\leq q \leq p$. Use  $u_1-u_2$ as a test function to get
$$\|\nabla (u_1-u_2)\|_p^p \leq C \|u_1-u_2\|_{\frac{m}{m-1}}\| f_1-f_2 \|_m $$
in view of the H\"older inequality and then $\| \nabla (u_1-u_2) \|_p \leq C \|f_1-f_2\|_m^{\frac{1}{p-1}}$ by the Sobolev embedding Theorem in view of $\frac{m}{m-1}<p^*$.  Notice that such last argument works as well as $m= \frac{Np}{Np-N+p}$ for $p<N$ since $\frac{Np}{Np-N+p}>1$ in this case.

\medskip \noindent Fix now $m=1$ and let $u_1,u_2 \in W^{1,p}_g(\Omega)$ be solutions of  $-\Delta_p u_i=f_i$, $i=1,2$,  in $\Omega$ with $\|f_i\|_1 \leq C_0$.  Use $T_{0,\epsilon}(u_1-u_2)$,  $T_{k,l}$ given by \eqref{trunc},  as a test function to get
\begin{equation} \label{18466}
 \int_{\{  |u_1-u_2|\leq \epsilon \} }  |\nabla (u_1-u_2)|^p \leq C \epsilon \|f_1-f_2\|_1 \leq 2CC_0 \epsilon
\end{equation} 
in view of \eqref{aboveestimate} and $p\geq 2$.  Given $1\leq q <\bar q$, by \eqref{1005} and H\"older's inequality \eqref{18466} implies
\begin{eqnarray} \label{184666}
 \int_\Omega |\nabla (u_1-u_2)|^q & \leq& C' \epsilon^{\frac{q}{p}}+(\int_{\{  |u_1-u_2|>\epsilon \} }  |\nabla (u_1-u_2)|^s)^{\frac{q}{s}}|\{  |u_1-u_2|>\epsilon \}|^{\frac{s-q}{s}} \nonumber\\
&\leq & C(\epsilon^{\frac{q}{p}}+\{  |u_1-u_2|>\epsilon \}|^{\frac{s-q}{s}})
\end{eqnarray} 
for some $q<s<\bar q$ in view of $\bar q<p$.  Since $g$ is $p-$harmonic in $\Omega$, taking now a sequence of solutions $u_n \in W^{1,p}_g(\Omega)$ to $-\Delta_p u_n=f_n$ in $\Omega$ with $\sup_n \|f_n\|_1 <+\infty$, by the first part we know that $u_n-g$ is bounded in $W_0^{1,q}(\Omega)$ and then, up to a subsequence,  we have that $u_n \rightharpoonup u$ in $W_g^{1,q}(\Omega)$ for all $1\leq q <\bar q$ and strongly in $L^s(\Omega)$ for all $1 \leq s <\bar q^*$. Applying \eqref{184666} to $u_n-u_m$ it is easily seen that $u_n$ is a Cauchy sequence in $W_g^{1,q}(\Omega)$  and then converges to $u$ in $W_g^{1,q}(\Omega)$ for all $1\leq q <\bar q$. The proof is complete.
\end{proof}

\noindent Let us push further the analysis in Lemma \ref{lemmaGj} towards an $L^\infty$-estimate when $p>\frac{N}{2}$. 
\begin{proposition}  \label{sup-estimate}
Let $2\leq p \leq N$ with $p> \frac{N}{2}$ and $M>0$. Then there exists $C>0$ so that $\|u_1-u_2\|_\infty \leq C$ for any pair $u_i \in W^{1,p}_{g_i}(\Omega)$, $i=1,2$, of solutions to 
\begin{equation} \label{1725}
-\Delta_p u_i -\lambda^i |u_i|^{p-2}u_i = f \quad \text{ in } \Omega,
\end{equation}
where $\|f \|_1+\displaystyle \sup_{i=1,2} \left[\frac{1}{(\lambda_1-\lambda^i)_+} +\|g_i \|_\infty \right]\leq M$ and $g_1$, $g_2$ satisfy \eqref{1440}.
\end{proposition}
\begin{proof} By  Lemma \ref{lemmaGj} we get an universal bound on $\|f+\lambda^i |u_i|^{p-2}u_i  \|_1$.  Since $g_i$ is $p-$harmonic function in $\Omega$, Lemma \ref{lemmaaggiunto} and the Sobolev embedding Theorem provide an universal bound on $u_i-g_i$ in $W_0^{1,q}(\Omega)$ for all $1\leq q<\bar{q}$ and $u_i$ in $L^q(\Omega)$ for all $1\leq q<\bar q^*$.  Since $\frac{\bar q^*}{p-1}>\frac{N}{p}$ thanks to $p>\frac{N}{2}$,  we can find $q_0>\frac{N}{p}$ so that $\hat f=\lambda^1 |u_1|^{p-2}u_1-\lambda^2 |u_2|^{p-2}u_2$ satisfies 
\begin{equation} \label{0160999}
\| \hat f \|_{q_0}\leq C
\end{equation} 
for some universal $C>0$. Thanks to \eqref{1725} we can write
\begin{equation} \label{1756}
\left\{ \begin{array}{ll} -\Delta_p u_1+\Delta_p u_2=\hat f &\hbox{in }\Omega\\
u_1-u_2=g_1-g_2 &\hbox{on }\partial \Omega. \end{array} \right.
\end{equation}
Since $q_0>\frac{N}{p}$ let us fix $\beta_0>0$ sufficiently small so that $p_0:=\frac{q_0(\beta_0-1+p)}{q_0-1}<\bar q^*$.  Set $u=u_1-u_2$, $C_0=\|g_1\|_\infty+\|g_2\|_\infty$ and define $\Psi(s)=[T_{0,l}(s  \mp C_0)_\pm+\epsilon]^\beta-\epsilon^\beta$, with $l,\epsilon>0$ and $\beta \geq \beta_0$, where $T_{k,l}$ is given by \eqref{trunc}. 
Notice that $l <+\infty$ and $\epsilon>0$ guarantee the boundedness and the differentiability of $\Psi$ in $\mathbb{R}$, respectively.  Use $\Psi(u) \in W^{1,p}_0(\Omega)$ as a test function in \eqref{1756} to get
\begin{eqnarray} \label{181666}
 \beta \int_{\{ (u \mp C_0)_\pm  \leq l \}}  [T_{0,l}(u \mp C_0)_\pm+\epsilon]^{\beta-1} (|\nabla u_2|+ |\nabla u| )^{p-2}  |\nabla u|^2  \leq C \int_\Omega |\hat f| [T_{0,l}(u \mp C_0)_\pm+\epsilon]^\beta
\end{eqnarray}
in view of \eqref{aboveestimate}.  Since $p \geq 2$ ,  by H\"older's inequality with exponents $\frac{q_0(\beta-1+p)}{(q_0-1)(p-1) }$, $q_0$ and $\frac{q_0 (\beta-1+p)}{(q_0-1)\beta}$ estimate \eqref{181666} implies the following estimate:
\begin{eqnarray*}
 \frac{\delta p^p \beta}{(\beta-1+p)^p} \int_\Omega |\nabla w_{l,\epsilon}|^p \leq  |\Omega|^{\frac{(q_0-1)(p-1)}{q_0(\beta-1+p)}} \| | \hat f  \|_{q_0}
\|w_{l,\epsilon}\|_{\frac{p q_0}{q_0-1}}^{\frac{\beta p}{\beta-1+p}} \leq C \|w_\epsilon \|_{\frac{p q_0}{q_0-1}}^{\frac{\beta p}{\beta-1+p}}
\end{eqnarray*}
for some $C>0$,  where 
$$w_{l,\epsilon}=[T_{0,l}(u \mp C_0)_\pm +\epsilon]^\frac{\beta-1+p}{p}, \quad  w_\epsilon=[(u \mp C_0)_\pm+\epsilon]^\frac{\beta-1+p}{p}, \quad w=(u \mp C_0)_\pm^\frac{\beta-1+p}{p}.$$
By the Sobolev embedding Theorem on $w_{l,\epsilon} -\epsilon^\frac{\beta-1+p}{p} \in W^{1,p}_0(\Omega)$ and the Fatou convergence Theorem as $l \to +\infty$ we deduce that
\begin{eqnarray} \label{1817}
\|w_\epsilon- \epsilon^\frac{\beta-1+p}{p} \|_{p^*} \leq C (\beta-1+p) \|w_\epsilon \|_{\frac{p q_0}{q_0-1}}^{\frac{\beta}{\beta-1+p}} 
\end{eqnarray}
for some $C>0$ provided the R.H.S. is finite, where $p^*=\frac{Np}{N-p}$ if $p<N$ and $p^* \in (\frac{pq_0}{q_0-1},+\infty)$ if $p=N$.  By using again the Fatou convergence Theorem on the L.H.S. and the Lebesgue convergence Theorem on the R.H.S. in \eqref{1817}, as $\epsilon \to 0$ we deduce that
$$\|w \|_{p^*} \leq C (\beta-1+p) \|w \|_{\frac{p q_0}{q_0-1}}^{\frac{\beta}{\beta-1+p}}$$ 
for some $C>0$, provided $\|w\|_{\frac{p q_0}{q_0-1}}<+\infty$. By the definition of $w$ and taking the $\frac{p}{\beta-1+p}-$power we then deduce that 
\begin{eqnarray*}
\| (u \mp C_0)_\pm    \|_{\frac{(\beta-1+p)p^*}{p}} \leq [C (\beta-1+p)]^{\frac{p}{\beta-1+p}} 
\|(u \mp C_0)_\pm  \|_{\frac{q_0(\beta-1+p )}{q_0-1}}^{\frac{\beta}{\beta-1+p}} ,
\end{eqnarray*}
or equivalently
\begin{eqnarray} \label{1513}
\| (u \mp C_0)_\pm   \|_{\kappa \mu } \leq [ C \frac{q_0-1}{q_0} \mu ]^{\frac{p q_0}{\mu(q_0-1)}} 
\|(u \mp C_0)_\pm  \|_{\mu}^{1-\frac{(p-1)q_0}{\mu(q_0-1)}},
\end{eqnarray}
where $\mu=\frac{q_0(\beta-1+p)}{q_0-1}$ and $\kappa=\frac{(q_0-1)p^*}{pq_0}>1$ in view of $q_0>\frac{N}{p}$. Setting $\mu_j= \kappa^j p_0$, we can perform $j+1$ iterations of  \eqref{1513} to get
\begin{eqnarray*}
\| (u \mp C_0)_\pm    \|_{\mu_{j+1}} &\leq& [C (\beta_0-1+p) \kappa^j ]^{\frac{p}{(\beta_0-1+p)\kappa^j}} 
\|(u  \mp C_0)_\pm  \|_{\mu_j}^{1-\frac{p-1}{(\beta_0-1+p)\kappa^j}}\leq \ldots\\
&\leq & [C (\beta_0-1+p)+1]^{\displaystyle \sum_{s=0}^j \frac{1}{\kappa^s}} \kappa^{\displaystyle \sum_{s=0}^j \frac{s}{\kappa^s}}
\|(u \mp C_0)_\pm  \|_{p_0}^{\displaystyle  \prod_{s=0}^j (1-\frac{p-1}{(\beta_0-1+p) \kappa^s})} 
\end{eqnarray*}
in view of $[C (\beta_0-1+p)+1] \kappa^j \geq 1$ and $1-\frac{p-1}{(\beta_0-1+p) \kappa^s}\leq 1$. By letting $j \to +\infty$ we deduce that
$$\| (u  \mp C_0)_\pm     \|_{\infty}  \leq C' \|(u  \mp C_0)_\pm  \|_{p_0}^{\theta_0} \leq C'_M$$
in view of
$$\theta_0:=\prod_{s=0}^\infty (1-\frac{p-1}{(\beta_0-1+p)\kappa^s})<+\infty, \quad \sum_{s=0}^\infty \frac{1}{\kappa^s}+\sum_{s=0}^\infty \frac{s}{\kappa^s}<+\infty.$$
In conclusion, $\|u_1-u_2\|_\infty \leq C_M'+C_0 \leq C_M$ and the proof is complete.
\end{proof}
The aim now is to extend Proposition \ref{sup-estimate} to $H_\lambda$ as a solution of \eqref{1123} (to be compared with \eqref{1756}) and to include the case $1<p< 2$. Since it is no longer a matter of universal estimates, the argument is potentially simpler but the singular character of equation \eqref{1123} has to be controlled thanks to the assumption $\nabla H_\lambda \in L^{\bar q}(\Omega)$. For later convenience, let us write the following result in a sufficiently general way.  
\begin{lemma} \label{corollaryweakformulation}
Let $1<p\leq N$ and $u \in W^{1,p}_{\hbox{loc}}(\Omega \setminus \{0\})$ be a solution of 
\begin{equation} \label{mathcalHproblem_lemmaweakformulation}
- \Delta_p (\mathit \Gamma + u) + \Delta_p \mathit  \Gamma = f \quad \text{ in } \Omega \setminus \{0\}
\end{equation}
with $f \in L^1(\Omega)$, $\nabla u \in L^{\bar q}(\Omega)$ and
\begin{equation} \label{01648}
\begin{array}{rl}
\displaystyle  \frac{1}{C}|\nabla \Gamma|\leq |\nabla \mathit  \Gamma| \leq C |\nabla \Gamma|&\hbox{if }1<p< 2\\
\displaystyle  |\nabla \mathit  \Gamma| \leq C |\nabla \Gamma|&\hbox{if }p \geq 2
\end{array}
\end{equation}
in $\Omega$ for some $C>1$. Let $ \eta \in C^1(\bar \Omega)$ and $\Psi \colon \R \to \R$ be a bounded monotone Lipschitz function. Assuming either $\eta=0$  or $\Psi(u)=0$ on $\partial \Omega$, then there holds
$$\int_\Omega \eta^2 |\Psi'(u)| (|\nabla \mathit \Gamma|+ |\nabla u| )^{p-2}  |\nabla u|^2 \leq C \Big( \int_\Omega |\eta| |\nabla \eta|  |\Psi(u)| (|\nabla \mathit \Gamma| + |\nabla u| )^{p-2} |\nabla u| 
+ \int_\Omega \eta^2 |f| |\Psi(u)|  \Big)$$
for some $C>0$. \end{lemma}
\begin{proof}
Consider a sequence $\eta_\epsilon \in C^1(\bar \Omega)$ so that
\begin{equation} \label{etaepsilonprop}
\eta_\epsilon=\eta \text{ in } \Omega\setminus B_{\epsilon}(0), \quad \eta_\epsilon=0 \text{ in } B_\frac{\epsilon}{2}(0), \quad |\eta_\epsilon|+\epsilon |\nabla \eta_\epsilon| \leq C \text{ in } B_{\epsilon}(0)\setminus B_\frac{\epsilon}{2}(0)
\end{equation}
for some $C>0$. Since $\eta_\epsilon^2  \Psi(u)$ vanishes in $B_\frac{\epsilon}{2}(0)$ and on $\partial \Omega$, it can be used a test function in \eqref{mathcalHproblem_lemmaweakformulation}:
\begin{eqnarray} \label{1925}
\int_\Omega \eta_\epsilon^2  |\Psi' (u)| (|\nabla \mathit \Gamma|+ |\nabla u|)^{p-2}   |\nabla u|^2 
\leq C \int_\Omega \Big[ |\eta_\epsilon| |\nabla \eta_\epsilon| (|\nabla \mathit  \Gamma| + |\nabla u|)^{p-2} |\nabla u| + \eta_\epsilon^2 |f|\Big] |\Psi(u)| 
\end{eqnarray}
for some $C>0$ since $\Psi'$ has given sign. We have used here \eqref{aboveestimate} and the estimate
$$\Big| |x+y|^{p-2} (x+y) - |x|^{p-2} x \Big| =(|x| + |y|)^{p-2} O(|y|).$$
Since $(|\nabla \mathit  \Gamma|+ |\nabla u|)^{p-2} = O(|\nabla \Gamma|^{p-2}+ |\nabla u|^{p-2})$ in view of \eqref{01648}, by the H\"older inequality we have that 
\begin{eqnarray}
&& \int_{B_\epsilon(0) \setminus B_{\frac{\epsilon}{2}}(0)} |\eta_\epsilon| |\nabla \eta_\epsilon||\Psi(u)| (|\nabla \mathit  \Gamma|+ |\nabla u|)^{p-2} |\nabla u| \leq C  \int_{B_\epsilon(0) \setminus B_{\frac{\epsilon}{2}}(0)} (\frac{|\nabla u|}{\epsilon^{\frac{N(p-1)-(N-1)}{p-1}}} + \frac{|\nabla u|^{p-1}}{\epsilon}  )  \nonumber \\
&& \leq C \left[ (\int_{B_\epsilon(0) \setminus B_{\frac{\epsilon}{2}}(0)} |\nabla u|^{\bar{q}})^\frac{1}{\bar{q}}+ (\int_{B_\epsilon(0) \setminus B_{\frac{\epsilon}{2}}(0)} |\nabla u|^{\bar{q}})^\frac{N-1}{N}\right] \to 0 \label{1951}
\end{eqnarray}
as $\epsilon \to 0$, in view of $\|\Psi\|_\infty<+\infty$ and $\nabla u \in L^{\bar q}(\Omega)$. By inserting \eqref{1951} into \eqref{1925} and by using the Lebesgue convergence Theorem for $\int_\Omega \eta_\epsilon^2 |f||\Psi(u)|$ we get the validity of Lemma \ref{corollaryweakformulation} in view of the monotone convergence Theorem.
\end{proof}

We are now ready to complete the proof of Theorem \ref{theoremexistenceG} by establishing $L^\infty-$bounds on $H_\lambda$.
\begin{theorem} \label{Hbounded}
Let $1<p \leq N$ and assume either $\lambda=0$ or $\lambda \not=0$ and $p\geq 2$ with $p> \frac{N}{2}$. Then $H_\lambda=G_\lambda-\Gamma \in L^\infty(\Omega)$, where $G_\lambda$ is any solution to \eqref{353} satisfying \eqref{122}.
\end{theorem}
\begin{proof}  By \eqref{353} the function $u=H_\lambda$ solves \eqref{mathcalHproblem_lemmaweakformulation} with $\mathit \Gamma=\Gamma$ and $f=\lambda G_\lambda^{p-1}$. Given $0<\beta_0<1$ to be fixed later, by Lemma \ref{corollaryweakformulation} with $\eta=1$ and $\Psi(s)=[T_{0,l}(s \mp C_0)_\pm+\epsilon]^\beta-\epsilon^\beta$, with $l,\epsilon> 0$, $\beta \geq \beta_0$, $C_0=\|g\|_\infty+\|\Gamma\|_{\infty,\partial \Omega}$ and $T_{k,l}$ given by \eqref{trunc},  we get that
\begin{eqnarray} \label{1816}
 \beta \int_{\{ (u \mp C_0)_\pm  \leq l \}}  [T_{0,l}(u \mp C_0)_\pm+\epsilon]^{\beta-1} (|\nabla \Gamma|+ |\nabla u| )^{p-2}  |\nabla u|^2  \leq C \int_\Omega |f| [T_{0,l}(u \mp C_0)_\pm+\epsilon]^\beta
\end{eqnarray}
in view of $\Psi(u)=0$ on $\partial \Omega$ thanks to $H_\lambda=g-\Gamma$ on $\partial \Omega$. 

\medskip \noindent Let us first consider the case $\lambda=0$. Then $f=0$ and the choice $\beta=1$ in \eqref{1816} gives
$$\int_{\Omega}  (|\nabla \Gamma|+ |\nabla u| )^{p-2}  |\nabla T_{0,l}(u \mp C_0)_\pm|^2 \leq 0.$$
Then $T_{0,l}(u \mp C_0)_\pm=0$ a.e. in $\Omega$ for any $l>0$, which implies $|H_0| \leq C_0$ a.e. in $\Omega$.

\medskip \noindent Consider now the case $\lambda \not=0$ and assume $p\geq 2$ with $p> \frac{N}{2}$. Since $\nabla G_\lambda=\nabla \Gamma+\nabla H_\lambda \in L^q(\Omega)$ for all $1\leq q <\bar q$ in view of \eqref{122}, by the Sobolev embedding Theorem $G_\lambda \in L^q(\Omega)$ for all $1\leq q <\bar q^*$ and in particular $f$ satisfies
\begin{equation} \label{01609}
\| f \|_{q_0}< \infty
\end{equation} 
for some $q_0>\frac{N}{p}$ in view of $p>\frac{N}{2}$.  

\medskip \noindent Notice that \eqref{1816}-\eqref{01609} are the analogue of \eqref{0160999} and \eqref{181666}, and then the argument now goes exactly as in the proof of Proposition \ref{sup-estimate}.
\end{proof}

For the case $g=0$ let us collect here some useful facts which will be used in the next two sections. Given $1<p<N$, an important ingredient is given by the estimate
\begin{equation} \label{conditionnablaH_phdthesis}
|\nabla H_\lambda| =O(|\nabla \Gamma|) \quad \hbox{ in }\Omega
\end{equation} 
for any solution $G_\lambda= \Gamma+H_\lambda$ of \eqref{353}$_{g=0}$. Indeed, by \cite{serrin65} any solution $G_\lambda$ of \eqref{353}$_{g=0}$ satisfies
\begin{equation} \label{boundGammaG}
\frac{\Gamma}{C}  \leq G_\lambda \leq C \Gamma \quad \hbox{ in }B_{2R_0}(0)
\end{equation}
for some $C>1$, where $R_0=\frac{1}{4}\hbox{dist}(0,\partial \Omega)$. For $0<R \leq R_0$ consider the scaling $G_{\lambda,R}(y)= R^\frac{N-p}{p-1} G_\lambda(Ry)$ of $G_\lambda$ in $\Omega_R=\frac{\Omega}{R}$ which satisfies
\begin{equation} \label{GRproblem_phdthesis}
\begin{cases}
-\Delta_p G_{\lambda,R} - \lambda R^p G_{\lambda,R}^{p-1}= \delta_0 &\text{in } \Omega_R  \\
G_{\lambda,R} \geq 0 &\text{in } \Omega_R \\
G_{\lambda,R} =0 &\text{on } \partial \Omega_R.
\end{cases}
\end{equation}
Since $\Gamma_R(y)= R^\frac{N-p}{p-1} \Gamma(Ry)=  \Gamma(y)$ in view of $1<p<N$, we have that condition \eqref{boundGammaG} is scaling invariant:
\begin{equation} \label{boundGammaGR}
\frac{\Gamma}{C} \leq G_{\lambda,R} \leq C \Gamma \quad \text{ in } B_{\frac{2 R_0}{R}}(0).
\end{equation}
Since $G_{\lambda,R}$ is uniformly bounded in $ L^\infty_{\hbox{loc}}(B_2(0) \setminus \{ 0 \})$ thanks to \eqref{boundGammaGR}, elliptic estimates \cite{dib,tolksdorf} for \eqref{GRproblem_phdthesis} imply that 
$$G_{\lambda,R} \hbox{ uniformly bounded in }C^{1,\alpha}_{\hbox{loc}}(B_2(0) \setminus \{ 0 \})$$
for some $\alpha \in (0,1)$. Since in particular $\| \nabla G_{\lambda,R} \|_{\infty, \partial B_1(0)} \leq C$, setting $H_{\lambda,R}(y)= R^\frac{N-p}{p-1} H_{\lambda}(Ry)$ we deduce that $\| \nabla H_{\lambda,R} \|_{\infty, \partial B_1(0)} \leq C'$ in view of $\nabla G_{\lambda,R}= \nabla \Gamma + \nabla H_{\lambda,R}$, which can be re-written as
\begin{equation} \label{1045}
|\nabla H_{\lambda}| \leq \frac{C'}{|x|^\frac{N-1}{p-1}} = C |\nabla \Gamma| \quad \hbox{ on }\partial B_R(0)  
\end{equation}
for all $0<R\leq \frac{1}{4} \hbox{dist}(0,\partial \Omega)$. Away from the origin $\nabla H_{\lambda}$ is bounded thanks to \cite{dib,lieberman,tolksdorf} and $|\nabla \Gamma|$ is bounded from below, and then estimate \eqref{conditionnablaH_phdthesis} follows by \eqref{1045}. Moreover, notice that for $1<p\leq N$ there holds
\begin{equation} \label{1530}
\|H_{\lambda}\|_\infty<+\infty \quad \Rightarrow \quad |\nabla H_{\lambda}(x)| =o( |\nabla \Gamma(x)|) \quad \hbox{as }x \to 0.
\end{equation}
Indeed, for $1<p<N$ we have that $\|H_{\lambda,R}\|_{\infty,\Omega_R} \to 0$ and then $\|\nabla H_{\lambda,R}\|_{\infty,\partial B_1(0)} \to 0$ as $R \to 0$, which provides the validity of \eqref{1530}. When $p=N$ the function $G_{\lambda,R}(y)= G_\lambda(Ry)+(N\omega_N)^{-\frac{1}{N-1}} \log R=\Gamma(y)+H_\lambda(Ry)$ is uniformly bounded in $ L^\infty_{\hbox{loc}}(\mathbb{R}^N \setminus \{ 0 \})$ and satisfies
$$-\Delta_N G_{\lambda,R} - \lambda R^N \Big[G_{\lambda,R}-(N\omega_N)^{-\frac{1}{N-1}} \log R \Big]^{N-1}= \delta_0 \quad \text{ in } \Omega_R. $$ 
We argue as above to show that, up to a subsequence, $H_{\lambda,R}(y)=H_{\lambda}(Ry) \to H_0$ in $C^1_{\hbox{loc}}(\mathbb{R}^N \setminus \{ 0 \})$ as $R\to 0$, where $\|H_0\|_\infty<+\infty$ and $\Gamma+H_0$ is a $N-$harmonic function in $\mathbb{R}^N\setminus \{0\}$. It follows that $H_0$ is a constant function, see for example Lemma 4.3 in \cite{Esp1}. Since this is true along any such subsequence, then $\nabla H_{\lambda,R} \to 0$  in $C_{\hbox{loc}}(\mathbb{R}^N \setminus \{ 0 \})$ as $R\to 0$ and \eqref{1530} does hold also in the case $p=N$.

\medskip \noindent Once we have $\delta |\nabla \Gamma|^{p-2} \leq (|\nabla \Gamma|+|\nabla H_\lambda|)^{p-2}$ for $1<p<2$ in view of \eqref{conditionnablaH_phdthesis}, it becomes clear the usefulness of the following weigthed Sobolev inequalities of Caffarelli-Kohn-Nirenberg type \cite{CKN}: given $1<p<2$, there exists $C>0$ so that
\begin{equation} \label{1335}
\left( \int_{\R^N} |\nabla \Gamma|^{p-2} |u|^\frac{2(N-2+p)}{N-p} \right)^\frac{N-p}{N-2+p} \leq C \int_{\R^N} |\nabla \Gamma|^{p-2} |\nabla u|^2 
\end{equation} 
for any compactly supported $u \in L^\infty (\mathbb{R}^N)$ with $\int_{\R^N} |\nabla \Gamma|^{p-2} |\nabla u|^2 <+\infty$. Valid in $C_0^\infty (\R^N)$,  \eqref{1335} can be first extended to $W^{1,2}-$functions with compact support in view of $|\nabla \Gamma|^{p-2} \in L^\infty_{\hbox{loc}}(\mathbb{R}^N)$ and then to compactly supported $u \in L^\infty (\mathbb{R}^N)$ with $\int_{\R^N} |\nabla \Gamma|^{p-2} |\nabla u|^2<+\infty$ through the sequence $\eta_\epsilon u \in W^{1,2}(\mathbb{R}^N)$, $\eta_\epsilon$ being given by \eqref{etaepsilonprop} with $\eta=1$ in $\mathbb{R}^N$, since
$$\lim_{\epsilon \to 0} \int_{\mathbb{R}^N} |\nabla \Gamma|^{p-2}|\nabla \eta_\epsilon|^2 u^2 \to 0.$$
For later convenience,  when either $2\leq p < N$ or $p=N\geq 3$ observe also the validity of the following inequality
\begin{equation} \label{01832}
\left( \int_{\R^N}  |u|^\frac{2N(p-1)}{N(p-1)-p} \right)^\frac{N(p-1)-p}{N(p-1)} \leq C \int_{\R^N} 
|x|^\frac{p-2}{p-1} |\nabla u|^2 
\end{equation} 
for any compactly supported $u \in L^\infty (\mathbb{R}^N)$ with $\int_{\R^N} |x|^\frac{p-2}{p-1} 
 |\nabla u|^2 <+\infty$. 

\section{Weak comparison principle and uniqueness results} \label{sectionuniquenessG}
This section is devoted to discuss the uniqueness part in Theorem \ref{mainth} when $2\leq p \leq N$ among solutions satisfying the natural condition \eqref{122}. When $\lambda=0$ maximum and comparison principle in weak or strong form are well known, see for example \cite{vazquez}, and have been extended in various forms to the case $\lambda<\lambda_1$ in connection with existence and uniqueness results, see \cite{CuTa,diazsaa,takac,fleckingertakac} just to quote a few.

\medskip \noindent To extend the previous uniqueness results to the singular situation, the crucial property is given by the convexity of the functional 
\begin{equation*}
I(w)= \begin{cases}  \displaystyle \int_\Omega |\nabla w^\frac{1}{p}|^p  \qquad &\text{ if }w \geq 0 \text{ and } \nabla(w^\frac{1}{p}) \in L^p(\Omega) \\
+ \infty &\text{ otherwise}.
\end{cases}
\end{equation*}
Proved in \cite{diazsaa} for $p>1$, a quantitative form is established here giving a positive lower bound for $I''$ when $2\leq p \leq N$, crucial to be applied on $\Omega_\epsilon=\Omega \setminus B_\epsilon(0)$ as $\epsilon \to 0$.
\begin{lemma} \label{lemmatecnicouniqueness}
Let $w \geq 0$ a.e. in $\Omega$ so that $\nabla(w^\frac{1}{p}) \in L^p(\Omega)$. Let $\phi$ be a direction so that $w_t= w+t\phi \geq 0$ a.e. in $\Omega$ and $\nabla(w_t^\frac{1}{p}) \in L^p(\Omega)$ for $t \geq 0$ small. Letting $\rho(w,\phi)$ be given in \eqref{derivataseconda2}, there hold
\begin{equation} \label{2007}
I'(w)[\phi] =\int_\Omega |\nabla w^\frac{1}{p}|^{p-2} \langle \nabla w^\frac{1}{p}, \nabla (w^\frac{1-p}{p} \phi) \rangle,\quad I''(w)[\phi,\phi]= \int_\Omega \rho(w, \phi) 
\end{equation}
with
\begin{eqnarray} \label{1349}
\rho(w,\phi) &\geq&  
 \frac{p-1}{p} (p^3-3p^2+5p-2)  |\nabla w^\frac{1}{p} |^p \left(\frac{\phi}{w}  -\frac{p(p^2-2p+2) \langle \nabla w,\nabla \phi \rangle}{(p^3-3p^2+5p-2)|\nabla w|^2} 
 \right)^2 \nonumber \\
&&+\frac{(p-1)(p-2)}{p(p^3-3p^2+5p-2)}  w^\frac{2(1-p)}{p} |\nabla w^\frac{1}{p} |^{p-2} 
|\nabla \phi|^2,
\end{eqnarray}
where $I'(w)[\phi] = \frac{d}{dt} I(w_t)\Big|_{t=0^+}$ and $I''(w)[\phi,\phi]=\frac{d}{dt} I'(w_t)[\phi] \Big|_{t=0^+}$.  \end{lemma}
\begin{proof}
Since $\frac{d}{dt} w_t^\frac{1}{p}=\frac{1}{p} w_t^\frac{1-p}{p}\phi$, we have that
$$I'(w_t)[\phi] = \int_\Omega |\nabla w_t^\frac{1}{p}|^{p-2} \langle \nabla w_t^\frac{1}{p}, \nabla (w_t^\frac{1-p}{p} \phi) \rangle,$$
providing, when evaluated at $t=0$, the validity of the first in formula \eqref{2007}. Differentiating once more in $t$ at $0^+$, we have that 
\begin{eqnarray} \label{derivataseconda0}
I''(w)[\phi,\phi]&=& (p-2) \int_\Omega | \nabla w^\frac{1}{p} |^{p-4}  \langle \nabla w^\frac{1}{p}, \nabla ( w^{\frac{1-p}{p}} \phi) \rangle^2 + \frac{1}{p} \int_{\Omega} |\nabla w^\frac{1}{p} |^{p-2} |\nabla (w^\frac{1-p}{p} \phi )|^2  \\
&&- \frac{p-1}{p} \int_\Omega |\nabla w^\frac{1}{p} |^{p-2} \langle \nabla w^\frac{1}{p}, \nabla ( w^\frac{1-2p}{p} \phi^2) \rangle . \nonumber
\end{eqnarray}
Writing $\langle \nabla w, \nabla \phi \rangle = \cos \alpha |\nabla w| |\nabla \phi|$ the first, second and third term in \eqref{derivataseconda0} produce, respectively,
\begin{eqnarray} \label{primoaddendo}
\int_\Omega | \nabla w^\frac{1}{p} |^{p-4} \langle \nabla w^\frac{1}{p}, \nabla ( w^{\frac{1-p}{p}} \phi) \rangle ^2  =  \int_\Omega  \frac{| \nabla w^\frac{1}{p} |^{p-2}}{w^\frac{2(p-1)}{p} }  \Big[  \frac{(p-1)^2}{p^2} \frac{|\nabla w |^2}{w^2}\phi^2  + \cos^2 \alpha  |\nabla \phi|^2   \\
 -\frac{2(p-1)}{p} \cos \alpha    \frac{|\nabla w|}{w} \phi |\nabla \phi| \Big], \nonumber
\end{eqnarray}
\begin{eqnarray} \label{secondoaddendo}
 \int_{\Omega} |\nabla w^\frac{1}{p} |^{p-2} |\nabla (w^\frac{1-p}{p} \phi )|^2  = \int_\Omega \frac{|\nabla w^\frac{1}{p} |^{p-2}}{w^\frac{2(p-1)}{p} } \Big[\frac{(p-1)^2}{p^2}   \frac{|\nabla w|^2}{w^2} \phi^2  +
  |\nabla \phi|^2  \\
- \frac{2(p-1)}{p}  \cos \alpha \frac{|\nabla w|}{w} \phi  |\nabla \phi| \Big],  \nonumber 
\end{eqnarray}
\begin{eqnarray} \label{terzoaddendo}
\int_\Omega |\nabla w^\frac{1}{p} |^{p-2} \langle \nabla w^\frac{1}{p}, \nabla ( w^\frac{1-2p}{p} \phi^2) \rangle = \int_\Omega \frac{|\nabla w^\frac{1}{p} |^{p-2}}{w^\frac{2(p-1)}{p} } \Big[  
-\frac{2p-1}{p^2}  \frac{|\nabla w|^2}{w^2} \phi^2 + \frac{2}{p}  \cos \alpha \frac{|\nabla w|}{w} \phi |\nabla \phi|\Big].
\end{eqnarray}
Collecting \eqref{primoaddendo}-\eqref{terzoaddendo}, the expression of \eqref{derivataseconda0} becomes $I''(w)[\phi,\phi]=\int_\Omega \rho(w,\phi)$, with
\begin{eqnarray} 
\rho(w,\phi) &=&  w^\frac{2(1-p)}{p} | \nabla w^\frac{1}{p} |^{p-2}  \Big[  C_1 \frac{|\nabla w |^2}{w^2}\phi^2  -C_2   \cos \alpha \frac{|\nabla w|}{w} \phi |\nabla \phi|+C_3 |\nabla \phi|^2 \Big] \label{derivataseconda2} 
\\
&=&  w^\frac{2(1-p)}{p} | \nabla w^\frac{1}{p} |^{p-2} 
\Big[  C_1 (\frac{|\nabla w |}{w}\phi  -\frac{C_2}{2C_1} \cos \alpha  |\nabla \phi|)^2 
+\frac{4C_1 C_3- C_2^2\cos^2 \alpha}{4 C_1}  |\nabla \phi|^2 \Big] \nonumber
\end{eqnarray}
by a square completion in view of $C_1>0$, where 
\begin{equation*}
C_1=\frac{p-1}{p^3} (p^3-3p^2+5p-2), \quad C_2= \frac{2(p-1)}{p^2} (p^2-2p+2), \quad
C_3= \frac{1}{p} + (p-2)\cos^2 \alpha.
\end{equation*}
Since 
$$4 \frac{p-1}{p^3}(p^3-3p^2+5p-2)(p-2)-\frac{4(p-1)^2}{p^4}(p^2-2p+2)^2=
-4 \frac{p-1}{p^4}(p^3-4p^2+8p-4)<0,$$
then $4C_1C_3-C_2^2 \cos^2 \alpha \geq  4 \frac{(p-1)^2(p-2)}{p^4}$ and \eqref{1349} follows by \eqref{derivataseconda2}. \end{proof}
As a first application, we deduce the validity of a weak comparison principle for positive solutions.
\begin{proposition} \label{wcp}
Let $2\leq p \leq N$ and $a ,f_1,f_2 \in L^\infty (\Omega)$. Let $u_i \in C^1(\bar \Omega)$, $i=1,2$, be solutions to 
\begin{equation} \label{1541}
-\Delta_p u_i-a u_i^{p-1}=f_i \quad \text{ in } \Omega
\end{equation}
so that 
\begin{equation} \label{01110}
u_i>0 \hbox{ in }\Omega, \quad \frac{u_1}{u_2} \leq C  \hbox{ near } \partial \Omega  
\end{equation}
for some $C>0$. If $f_1 \leq f_2$ with $f_2 \geq 0$ in $\Omega$ and $u_1 \leq u_2$ on $\partial \Omega$, then $u_1\leq u_2$ in $\Omega$.
\end{proposition}
\begin{proof} Setting $w_1=u_1^p$, $w_2=u_2^p$ and $\phi=(w_1-w_2)_+$, consider $w_s=s w_1 + (1-s) w_2$ for $s \in [0,1]$. Since
$$w_s+t\phi=u_2^p \Big[s(\frac{u_1}{u_2})^p+(1-s) +t \left((\frac{u_1}{u_2})^p-1 \right)_+ \Big],$$
by \eqref{01110} there exists $t_0>0$ small so that $w_s+t \phi \geq 0$ in $\Omega$ and $\nabla (w_s+t \phi)^{\frac{1}{p}} \in L^p(\Omega)$ for each $s \in [0,1]$ and $|t|\leq t_0$. Then we can apply \eqref{2007} at $s=0,1$ to get
\begin{eqnarray*}
I'(w_1)[\phi] -I'(w_2)[\phi]&=& \int_\Omega |\nabla w_1^\frac{1}{p}|^{p-2} \langle \nabla w_1^\frac{1}{p}, \nabla (w_1^\frac{1-p}{p} \phi) \rangle -\int_\Omega |\nabla w_2^\frac{1}{p}|^{p-2} \langle \nabla w_2^\frac{1}{p}, \nabla (w_2^\frac{1-p}{p} \phi) \rangle\\
&=& \int_\Omega |\nabla u_1|^{p-2} \langle \nabla u_1, \nabla \frac{\phi}{u_1^{p-1}} \rangle -\int_\Omega |\nabla u_2|^{p-2} \langle \nabla u_2 , \nabla \frac{\phi}{u_2^{p-1}}  \rangle.
\end{eqnarray*}
Since $\phi \in W^{1,p}_0(\Omega)$ we deduce that
\begin{eqnarray*}
I'(w_1)[\phi] -I'(w_2)[\phi]=\int_\Omega \left(\frac{f_1}{u_1^{p-1}}-\frac{f_2}{u_2^{p-1}} \right) (u_1^p-u_2^p)^+\leq 0
\end{eqnarray*}
in view of \eqref{1541} and $f_1 \leq f_2$ with $f_2 \geq 0$. Since
$$I'(w_1)[\phi] -I'(w_2)[\phi]= \int_0^1 I''(w_s)[w_1-w_2,\phi]  ds=\int_0^1 I''(w_s)[\phi,\phi]  ds$$
in view of $I''(w_s)[w_1-w_2,\phi]=I''(w_s)[\phi,\phi]$, by Lemma \ref{lemmatecnicouniqueness} $I''(w_s)[\phi,\phi]= \int_\Omega \rho(w_s, \phi)$ with $\rho(w_s, \phi)\geq 0$ thanks to \eqref{1349} when $p\geq 2$. Then, we deduce that $\rho(w_s,\phi)=0$ for all $s \in [0,1]$ and then
\begin{itemize}
\item $\nabla \phi=0$ in $\Omega$ if $p>2$
\item $\langle \nabla w_s,\nabla \phi \rangle=  \phi \frac{|\nabla w_s|^2}{w_s}$ if $p=2$, which implies $ \langle \nabla (w_1-w_2),\nabla \phi \rangle=  s \phi \frac{|\nabla (w_1-w_2)|^2}{w_s}$ for all $0\leq s \leq 1$.
\end{itemize}
In both cases $\nabla \phi=0$ in $\Omega$ and then $w_1 \leq w_2$ in $\Omega$, or equivalently $u_1 \leq u_2$ in $\Omega$.
\end{proof}
Finally, we use Lemma \ref{lemmatecnicouniqueness} to show the uniqueness part in Theorem \ref{mainth}.
\begin{theorem} \label{theoremuniquenessG}
Let $2\leq p \leq N$. If $\lambda< \lambda_1$ with $\lambda \not=0$ and $p>\frac{N}{2}$, problem \eqref{353}$_{g=0}$ has exactly one solution $G_\lambda$ so that $H_\lambda=G_\lambda-\Gamma$ satisfies \eqref{122}. Moreover, if $H_\lambda \in C(\Omega)$ for all $\lambda<\lambda_1$, then the map $\lambda  \in (-\infty,\lambda_1) \to H_\lambda(x)$ is strictly increasing at any given $x \in \Omega$. 
\end{theorem}
\begin{proof} We follow the same argument as in the proof of Proposition \ref{wcp}. Letting $G_1$ and $G_2$ be two solutions of \eqref{353}$_{g=0}$ satisfying \eqref{122}, by elliptic regularity theory \cite{dib,lieberman,serrin,tolksdorf} we know that $G_i \in C^{1,\alpha}(\bar \Omega \setminus \{0\})$, $i=1,2$, for some $\alpha>0$. By \cite{serrin65} we know that $G_i$, $i=1,2$, satisfies \eqref{boundGammaG} and by the strong maximum principle \cite{vazquez} $\partial_\nu G_i <0$, $i=1,2$, on $\partial \Omega$, where $\nu$ denotes the outward unit normal vector. Set $w_1=G_1^p$, $w_2=G_2^p$, $\phi=w_1-w_2$ and $w_s=s w_1 + (1-s) w_2$ for $s \in [0,1]$. We have that for each $s \in [0,1]$ there hold $w_s+t \phi \geq 0$ in $\Omega$ and $\nabla (w_s+t \phi)^{\frac{1}{p}} \in L^p(\Omega)$ for $t$ small, in view of the properties of $G_1$ and $G_2$. Letting $I_\epsilon$ be the functional $I$ defined on $\Omega_\epsilon= \Omega \setminus B_\epsilon(0)$, by \eqref{2007} at $s=0,1$ we have that
\begin{eqnarray*}
I_\epsilon'(w_1)[\phi] -I_\epsilon'(w_2)[\phi]&=& \int_{\Omega_\epsilon} |\nabla G_1|^{p-2} \langle \nabla G_1, \nabla \frac{\phi}{G_1^{p-1}} \rangle -\int_\Omega |\nabla G_2|^{p-2} \langle \nabla G_2 , \nabla \frac{\phi}{G_2^{p-1}}  \rangle\\
&=& \int_{\partial B_\epsilon(0)} ( \frac{|\nabla G_2|^{p-2} \partial_\nu G_2}{G_2^{p-1}} - \frac{|\nabla G_1|^{p-2} \partial_\nu G_1}{G_1^{p-1}}) (G_1^p - G_2^p)
\end{eqnarray*}
in view of $\phi=0$ on $\partial \Omega$ and the equation  \eqref{353}$_{g=0}$ satisfied by $G_1$, $G_2$. Notice that
$$I_\epsilon '(w_1)[\phi] -I_\epsilon'(w_2)[\phi]= \int_0^1 I_\epsilon''(w_s)[\phi,\phi]  ds$$
with $I_\epsilon''(w_s)[\phi,\phi]= \int_{\Omega_\epsilon} \rho(w_s, \phi)$ in view of Lemma \ref{lemmatecnicouniqueness}. Since $\rho(w_s, \phi)\geq 0$ when $p\geq 2$ in view of \eqref{1349}, by the Fatou convergence Theorem we deduce that
\begin{equation} \label{1037bis}
\int_0^1 ds \int_\Omega \rho(w_s,\phi) \leq \lim_{\epsilon \to 0} \int_{\partial B_\epsilon(0)} ( \frac{|\nabla G_2|^{p-2} \partial_\nu G_2}{G_2^{p-1}} - \frac{|\nabla G_1|^{p-2} \partial_\nu G_1}{G_1^{p-1}}) (G_1^p - G_2^p). \end{equation}
We claim that the R.H.S. in \eqref{1037bis} vanishes and then $\rho(w_s,\phi)=0$ for all $s \in [0,1]$, which implies, as already discussed in the proof of Proposition \ref{wcp}, $\nabla \phi=0$ in $\Omega$ and then $G_1 =G_2$ in $\Omega$.

\medskip \noindent  In order to prove the previous claim, for $i=1,2$ notice that $H_i=G_i-\Gamma \in L^\infty(\Omega)$ follows by Theorem \ref{Hbounded} in view of the assumption \eqref{122} for $G_i$. Once $H_i  \in L^\infty(\Omega)$, we have that $H_i$ satisfies \eqref{1530} and then
\begin{eqnarray} \label{1055}
G_i^q=\Gamma^q+O(\Gamma^{q-1}),\quad |\nabla G_i|^{p-2} \partial_\nu G_i=|\nabla \Gamma|^{p-2}\partial_\nu \Gamma+o(|\nabla \Gamma|^{p-1})
\end{eqnarray}
as $x \to 0$ for $q>0$. By \eqref{1055} we deduce that $G_1^p-G_2^p=O(\Gamma^{p-1})$ and
$$\frac{|\nabla G_i|^{p-2} \partial_\nu G_i}{G_i^{p-1}}= \frac{|\nabla \Gamma|^{p-2}\partial_\nu \Gamma}{\Gamma^{p-1}}+o(\frac{|\nabla \Gamma|^{p-1}}{\Gamma^{p-1}}),$$
which imply
$$\Big| \int_{\partial B_\epsilon(0)} ( \frac{|\nabla G_2|^{p-2} \partial_\nu G_2}{G_2^{p-1}} - \frac{|\nabla G_1|^{p-2} \partial_\nu G_1}{G_1^{p-1}}) (G_1^p - G_2^p)\Big| =o(\int_{\partial B_\epsilon(0)} |\nabla \Gamma|^{p-1}) =o(1)$$
as $\epsilon \to 0$, as claimed. 

\medskip \noindent  Finally, assume $H_\lambda \in C(\Omega)$ for all $\lambda<\lambda_1$ to have well defined values $H_\lambda(x)$ for all $x \in \Omega$ (at $x=0$ too) and take $\mu_1<\mu_2$. Letting $0\leq G_n^1, G_n^2 \in W_0^{1,p}(\Omega)$ be the solutions of \eqref{Gjproblem_boundedness_Lpnorm} corresponding to $\lambda=\mu_1$ and $\lambda=\mu_2$, respectively, by the proof of Theorem \ref{theoremexistenceG} recall that $G_{\mu_1}=\displaystyle \lim_{n \to +\infty} G_n^1$ and $G_{\mu_2}=\displaystyle  \lim_{n \to +\infty} G_n^2$ a.e. in $\Omega$, where $f_n \geq 0$ is a suitable smooth approximating sequence for the measure $\delta_0$. Since $G_n^i>0$ in $\Omega$ and $\partial_\nu G_n^i<0$ on $\partial \Omega$  by the strong maximum principle \cite{vazquez}, we can apply Proposition \ref{wcp} to get $G_n^1\leq G_n^2$ in view of $0\leq f_n \leq f_n+(\mu_2-\mu_1) (G_n^2)^{p-1}$ with $f_n,G_n^2 \in L^\infty(\Omega)$,  and then $G_{\mu_1}\leq G_{\mu_2}$ in $\Omega$ as $n \to +\infty$. Since 
$$-\Delta_p G_{\mu_1}=\mu_1 (G_{\mu_1})^{p-1}< \mu_2 (G_{\mu_2})^{p-1}=-\Delta_p G_{\mu_2} \quad \hbox{in }\Omega \setminus B_\epsilon(0),$$
apply once again the strong maximum principle \cite{vazquez} to deduce $G_{\mu_1}<G_{\mu_2}$ in $\Omega \setminus B_\epsilon(0)$ for all $\epsilon>0$, and the strict monotonicity is established in $\Omega \setminus \{0\}$. Given $0<\epsilon<\hbox{dist }(0,\partial \Omega)$, we can find $\eta \in C^1_0(\Omega)$ with $\eta=1$ in $B_\epsilon(0)$ and $\delta>0$ so that $H_{\mu_1}-H_{\mu_2}+\delta \leq 0 $ on $\hbox{supp}(\eta) \setminus  B_\epsilon(0)$. Observe that $u=H_{\mu_1}-H_{\mu_2}$ and $\mathit  \Gamma=\Gamma+H_{\mu_2}$ satisfy $\nabla u \in L^{\bar q}(\Omega)$, \eqref{01648} and
$$-\Delta_p (\mathit  \Gamma+u)+\Delta_p(\mathit  \Gamma)=f \quad \hbox{ in }\Omega \setminus \{0\}$$
with $f=\mu_1 (G_{\mu_1})^{p-1}-\mu_2 (G_{\mu_2})^{p-1} \leq 0$. We can apply a variant of Lemma \ref{corollaryweakformulation} with $\eta$ and $\Psi(u)=(u+\delta)_+$ to get
$$\int_\Omega \eta^2 |\nabla (u+\delta)_+|^p \leq C \int_\Omega |\eta| |\nabla \eta|  (u+\delta)_+ (|\nabla \mathit \Gamma| + |\nabla u| )^{p-2} |\nabla u|
 +\int_\Omega \eta^2 f (u+\delta)_+ \leq 0$$
and then $(u+\delta)_+=0$ in $B_\epsilon(0)$, providing $H_{\mu_1}-H_{\mu_2}\leq -\delta<0$ in $B_\epsilon(0)$ too. The proof is complete.
\end{proof}

\section{Harnack inequalities and H\"older continuity of $H_\lambda$ at the pole} \label{chapterholder}
In this section we will use the Moser iterative scheme in \cite{serrin} to establish local estimates for the solution $H_\lambda$ of \eqref{1123} at $0$, leading to an Harnack inequality for $H_\lambda+c$ which is the crucial tool to show H\"older estimates at $0$. The function $ \mathcal{H}(x)=R^\frac{N-p}{p-1} (\pm H_\lambda(Rx)+c)$, $0<R<\frac{1}{2}\hbox{dist }(0,\partial \Omega)$, satisfies
\begin{equation} \label{mathcalHproblem}
 -\Delta_p (\mathit \Gamma + \mathcal{H}) + \Delta_p \mathit \Gamma = \mathcal G \quad \text{ in } B_2(0) \setminus \{0\}
\end{equation}
in view of \eqref{1123}, where $\mathit \Gamma=\pm R^\frac{N-p}{p-1} \Gamma(Rx)$ with $\nabla \mathit \Gamma=\pm \nabla \Gamma$ and $\mathcal G=\pm \lambda R^N G_\lambda^{p-1}(Rx)$. Differently from Proposition \ref{sup-estimate} and Theorem \ref{Hbounded}, we need to perform homogeneuos estimates on $\mathcal H$ and to this aim for $2\leq p \leq N$ assume 
\begin{equation} \label{definitionk}
\Lambda=  \|  \mathcal G \|^\frac{1}{p-1}_{q_0,B_2(0)} <+\infty
\end{equation}
for some $q_0>\frac{N}{p}$. Consider the weight function $\rho=|\nabla \Gamma|^{p-2}$ when $1<p<2$, $\mathcal G=0$ and $\rho=1$ otherwise, and introduce the weighted integrals $\Phi_\rho(s,h)= \left( \int_{B_h(0)} \rho |u|^s  \right)^\frac{1}{s}$, $h,s>0$.
Define $\kappa$ as
\begin{equation} \label{defkappa}
 \kappa= \begin{cases}
\frac{N-2+p}{N-p} &\hbox{if }1<p<2 \hbox{ and }  \mathcal G=0 \\
 \frac{N(p-1)}{N(p-1)-p} &\hbox{if either } 2\leq p < N\hbox{ or }p =N \geq 3\\
2 &\hbox{if }  p=N=2.
 \end{cases}
 \end{equation}
We are now ready to establish the main estimates in the section.
\begin{proposition} \label{propositionboundmathcalH}
Let $\mathcal H \in L^\infty(B_2(0))$ be a solution of \eqref{mathcalHproblem} so that $\nabla \mathcal H\in L^{\bar q}(B_2(0))$, $\mathit \Gamma$ satisfies \eqref{01648} and \eqref{definitionk} holds. Assume $\mathcal G=0$, $|\nabla \mathcal H|\leq M |\nabla \mathit \Gamma|$ in $B_2(0)$ when $1<p<2$ and $\|\mathcal H\|_{\infty}+\Lambda \leq M$, 
$|x|^\frac{1}{p-1}\leq M |\nabla \mathit \Gamma|$ in $B_2(0)$ when $2\leq p \leq N$, for some $M>0$. Given $\mu \in \mathbb{R}\setminus \{0\}$, there exist $\nu,\beta\geq 0$ and $C>0$ so that the function $u= |\mathcal{H}| + \Lambda+ \epsilon$ satisfies
\begin{equation} \label{lemma1mathcalH3}
\pm \Phi_\rho(\kappa \mu, h_1) \leq \pm [C |\mu|^{\nu} (h_2-h_1)^{-\beta}]^\frac{1}{\mu} \Phi_\rho(\mu,h_2) 
\end{equation}
for all $1\leq h_1<h_2 \leq 2$ and $0<\epsilon \leq 1$, uniformly for $\mu$ away from $2-p$, $0$ and $1$, where $\kappa>1$ is given in \eqref{defkappa} and $\pm$ simply denotes the sign of $\mu$.
\end{proposition}
\begin{remark} The assumption $|x|^\frac{1}{p-1}\leq M |\nabla \mathit \Gamma|$ when $2\leq p\leq N$ is sufficiently general in order to establish the validity of Corollary \ref{boundHepsilon},  which will be used in a crucial way in \cite{AnEs2}. \end{remark}
\begin{proof}
Given $T_{k,l}$ in \eqref{trunc}, introduce the bounded monotone Lipschitz function 
$$\Psi(s)= \hbox{sign }s \left( [T_{0,l}(|s|+\Lambda+\epsilon)]^\beta-[T_{0,l}(\Lambda+\epsilon)]^\beta \right), \beta \in \mathbb{R} \setminus \{ 0 \}.$$ 
Let $\eta \in C_0^\infty (B_{h_2}(0))$ be a cut-off function so that $0\leq \eta \leq 1$, $\eta=1$ in $B_{h_1}(0)$ and $|\nabla \eta| \leq \frac{2}{h_2-h_1}$. Since $\eta=0$ on $\partial B_2(0)$ and $\nabla \mathcal H \in L^{\bar q}(B_2(0))$ we can apply Lemma \ref{corollaryweakformulation} to $\mathcal H$, solution of \eqref{mathcalHproblem}, to get
\begin{eqnarray} \label{1451}
\int \eta^2 |\Psi'(\mathcal H)| (|\nabla \mathit \Gamma|+ |\nabla \mathcal H| )^{p-2}  |\nabla \mathcal H|^2 &\leq & C  \int  \eta |\nabla \eta|  |\Psi(\mathcal H)| (|\nabla \mathit \Gamma| + |\nabla \mathcal H| )^{p-2} |\nabla \mathcal H| \\
&&+C   \int  \eta^2 |\mathcal G| |\Psi(\mathcal H)| \nonumber
\end{eqnarray}
for some $C>0$. Define $v=u^\frac{\beta+1}{2}$ and $w=u^\frac{\beta-1+p}{p}$ with $u=|\mathcal{H}| + \Lambda+ \epsilon$.
Since $\Psi'(\mathcal H)= \beta u^{\beta-1}$ and $|\Psi(\mathcal H)| \leq u^\beta$ for $l>M+1$, by \eqref{1451} we deduce that
\begin{eqnarray} \label{1550}
|\beta| \int  \eta^2 u^{\beta-1} (|\nabla \mathit \Gamma|+ |\nabla u| )^{p-2}  |\nabla u|^2 \leq C  \left( \int  \eta |\nabla \eta| u^\beta (|\nabla \mathit \Gamma| + |\nabla u| )^{p-2} |\nabla u|+ \int  \eta^2 |\mathcal G| u^\beta  \right)
\end{eqnarray}
in view of $|\nabla \mathcal H|=|\nabla u|$.

\medskip \noindent Consider first the case $1<p<2$, for which \eqref{1550} implies
\begin{eqnarray} \label{pmin25_phdthesis}
\int  \eta^2 |\nabla \mathit \Gamma|^{p-2} |\nabla v|^2 \leq C  \int  \eta |\nabla \eta|  |\nabla \mathit \Gamma| ^{p-2} v |\nabla v|
\end{eqnarray}
uniformly for $\beta$ away from $0$ in view of $|\nabla u|\leq M |\nabla \mathit \Gamma|$ in $B_2(0)$. Since
$$ C  \int  \eta |\nabla \eta|  |\nabla \mathit \Gamma| ^{p-2} v |\nabla v|
\leq \frac{1}{2} \int  \eta^2 |\nabla \mathit \Gamma|^{p-2} |\nabla v |^2 + C' \int |\nabla \eta|^2  |\nabla \mathit \Gamma|^{p-2} v^2$$
thanks to the Young inequality, we can re-write \eqref{pmin25_phdthesis} as
\begin{eqnarray} \label{pmin28_phdthesis}
\int  |\nabla \mathit \Gamma|^{p-2} |\nabla (\eta v)|^2 \leq C \int  |\nabla \eta|^2  |\nabla \mathit \Gamma| ^{p-2} v^2 .
\end{eqnarray}
Thanks to \eqref{01648} and making use of  \eqref{1335}, by \eqref{pmin28_phdthesis} we deduce for $\mu=\beta+1$ that
$$\pm \Phi_\rho (\kappa \mu,h_1) \leq \pm (\frac{C}{(h_2-h_1)^2})^\frac{1}{\mu} \Phi_\rho (\mu,h_2)$$
does hold uniformly for $\mu$ away from $1$, where $\kappa$ is given by \eqref{defkappa}. Observe that the $(\beta+1)-$th root of \eqref{pmin28_phdthesis} for $\beta<-1$ reverses the inequality causing the presence of $\pm$ in \eqref{lemma1mathcalH3}.

\medskip \noindent Consider now the case $2\leq p \leq N$. Since
\begin{eqnarray*} 
&& C  \int  \eta^\frac{p}{2} |\nabla \eta^\frac{p}{2}| u^\beta (|\nabla \mathit \Gamma|+|\nabla u|)^{p-2} |\nabla u|\\ 
&&\leq\frac{|\beta|}{4} \int \eta^p u^{\beta-1} (|\nabla \mathit \Gamma|+|\nabla u|)^{p-2} |\nabla u|^2 +\frac{C'}{|\beta|}\int |\nabla \eta|^2 u^{\beta+1} |\nabla \mathit \Gamma|^{p-2}
+\frac{C'}{|\beta|}\int \eta^{p-2} |\nabla \eta|^2 u^{\beta+1} |\nabla u|^{p-2}\\
&&\leq\frac{|\beta|}{2} \int \eta^p u^{\beta-1} (|\nabla \mathit \Gamma|+|\nabla u|)^{p-2} |\nabla u|^2 +\frac{C}{|\beta|}\int |\nabla \eta|^2 v^2 
+\frac{C}{|\beta|^{p-1}}\int |\nabla \eta|^p w^p 
\end{eqnarray*}
in view of the Young inequality, \eqref{01648} and $\displaystyle \sup_{B_2 \setminus B_1} |\nabla \Gamma|^{p-2} < + \infty$, by replacing $\eta$ with $\eta^\frac{p}{2}$ \eqref{1550} implies
\begin{eqnarray} \label{1803}
 \int  \eta^p |\nabla \mathit \Gamma|^{p-2}  |\nabla v|^2 
+ \frac{1}{|\beta|^{p-2}}\int  \eta^p  |\nabla w|^p \leq C  ( \int  |\nabla \eta|^2 v^2+ 
\frac{1}{|\beta|^{p-2}}\int |\nabla \eta|^p w^p 
+|\beta| \int  \eta^p |\mathcal G| u^\beta  )
\end{eqnarray}
uniformly for $\beta$ away from $1-p$ and $0$. Since $q_0>\frac{N}{p}$, fix $\alpha$ and $\gamma$ so that $\alpha \in (\frac{q_0}{q_0-1}, \frac{pq_0}{N-p})$ and $\frac{1}{\alpha}+\frac{1}{\gamma}=\frac{q_0-1}{q_0}$. By the  H\"older inequality with exponents $q_0$, $\gamma$ and $\alpha$ we have that
$$\int  \eta^p |\mathcal{G}|  u^\beta  \leq \frac{1}{\Lambda^{p-1}} \int  |\mathcal{G}|  (\eta w)^{\frac{p}{\gamma}+\frac{p(q_0+\alpha)}{\alpha q_0}} 
\leq \frac{1}{\Lambda^{p-1}} \|\mathcal{G}\|_{q_0,B_2(0)} 
\|\eta w\|_p^{\frac{p}{\gamma}} \|\eta w\|_{\frac{p(q_0+\alpha)}{q_0}}^{\frac{p(q_0+\alpha)}{\alpha q_0}}=\|\eta w\|_p^{\frac{p}{\gamma}} \|\eta w\|_{\frac{p(q_0+\alpha)}{q_0}}^{\frac{p(q_0+\alpha)}{\alpha q_0}}$$
in view of \eqref{definitionk} and then
\begin{eqnarray} \label{1p16}
C |\beta| \int  \eta^p |\mathcal{G}|  u^\beta &\leq& C' |\beta| \|\eta w\|_p^{\frac{p}{\gamma}} (
\|\eta \nabla w\|_p^{\frac{p(q_0+\alpha)}{\alpha q_0}}+\| w \nabla \eta\|_p ^{\frac{p(q_0+\alpha)}{\alpha q_0}})\\
& \leq &  \frac{1}{2|\beta|^{p-2}} \| \eta \nabla w\|_p^p +C''  |\beta|^\frac{\alpha q_0+(p-2)(q_0+\alpha)}{\alpha q_0-\alpha-q_0} \|\eta w\|_p^p +\frac{1}{|\beta|^{p-2}}\|w \nabla \eta\|_p^p \nonumber
\end{eqnarray}
by the Sobolev embedding Theorem in view of $(N-p)(q_0+\alpha)<N q_0$ and the Young inequality. Inserting \eqref{1p16} into \eqref{1803} we get that
\begin{eqnarray} \label{1909}
\int    |x|^\frac{p-2}{p-1}  |\nabla (\eta^{\frac{p}{2}} v)|^2  \leq C  \left( \int  |\nabla \eta|^2 v^2+|\beta|^\frac{\alpha q_0+(p-2)(q_0+\alpha)}{\alpha q_0-\alpha-q_0} \int \eta^p |w|^p   +\frac{1}{|\beta|^{p-2}} \int |\nabla \eta|^p |w|^p  \right)
\end{eqnarray}
in view of $|x|^\frac{1}{p-1} \leq M |\nabla \mathit \Gamma|$ in $B_2(0)$. Since $\|\mathcal H\|_{\infty}+\Lambda \leq M$ if $p\geq 2$, we have that $\| u\|_\infty \leq M+1$ when $0<\epsilon \leq 1$ and then $w^p=u^{\beta+1}u^{p-2} \leq (M+1)^{p-2} v^2$. By using the Sobolev embedding Theorem when $p=N=2$ or \eqref{01832} otherwise, for $\mu=\beta+1$ estimate \eqref{1909} gives that
$$\pm \Phi_1(\kappa \mu, h_1) \leq \pm
[C \frac{|\mu|^\frac{\alpha q_0+(p-2)(q_0+\alpha)}{\alpha q_0-\alpha-q_0}}{(h_2-h_1)^p}]^\frac{1}{\mu} \Phi_1(\mu,h_2)$$
does hold uniformly for $\mu$ away from $2-p$ and $1$, where $\kappa$ is given by \eqref{defkappa}. Estimate \eqref{lemma1mathcalH3} is then established in all the cases and the proof is complete. \end{proof}
Hereafter we specialize the argument to $\mathcal H=R^\frac{N-p}{p-1} (\pm H_\lambda(Rx)+c)$, $R>0$. Let us consider now the case $\beta=-1$ in the proof of Proposition \ref{propositionboundmathcalH} when $\mathcal H \geq 0$ and the result we have is the following.
\begin{proposition} \label{lemmaJN}
Let $1<p\leq N$ if  $\lambda=0$ and $p\geq 2$ with $p >\frac{N}{2}$ if $\lambda \not=0$. Assume $\frac{N}{p}<q_0<\frac{N}{N-p}$ if $\lambda \not= 0$ and $\mathcal H=R^\frac{N-p}{p-1} (\pm H_\lambda(Rx)+c) \geq  0$. There exist $R_0>0$ and $C>0$ so that $v=\log u$, where $u= \mathcal{H}+\Lambda+ \epsilon$ and $\epsilon>0$, satisfies
$$\fint_B |v - \bar{v}|  \leq C $$
for all open ball $B \subset B_1(0)$, $0<R\leq R_0$ and $0<\epsilon \leq 1$, where $\fint$ denotes an integral mean and $\bar{v}= \fint_{B} v $.
\end{proposition}
\begin{proof}
First of all, observe that $p\geq 2$ and $p>\frac{N}{2}$ imply $p^2 \geq 2p>N$. Let $B =B_h(x_0) \subset B_1(0)$. Since $|x_0|+h < 1$ implies $|x| \leq |x-x_0| + |x_0| < \frac{3}{2}h+|x_0|   < 2$ for all $x \in B_{\frac{3}{2}h}(x_0)$, we have that $B_{\frac{3}{2}h}(x_0) \subset B_2$. Let $\eta \in C_0^\infty(B_{\frac{3}{2}h}(x_0))$ be a cut-off function with $0\leq \eta \leq 1$, $\eta=1$ in $B_h(x_0)$ and $|\nabla \eta| \leq \frac{4}{h}$. Since $\mathcal H$ solves \eqref{mathcalHproblem} with $\nabla \mathit \Gamma=\pm \nabla \Gamma$ and $\mathcal G=\pm \lambda R^N G^{p-1}(Rx)$, we can apply Lemma \ref{corollaryweakformulation} with the bounded monotone Lipschitz function $\Psi(s)= \hbox{sign }s \left( [T_{0,l}(|s|+\Lambda+\epsilon)]^{-1}-[T_{0,l}(\Lambda+\epsilon)]^{-1} \right)$, for $l>\|\mathcal H \|_\infty+\Lambda+1$ and $T_{k,l}$ given by \eqref{trunc}, and a cut-off function $\eta_\delta=\eta (\delta +|x|^2)^{\frac{(N-1)(p-2)}{4(p-1)}-1} |x|^{\frac{5}{2}}$, $\delta>0$, to get
\begin{eqnarray*}
\int  \eta_\delta^2  |\nabla \Gamma|^{p-2}  |\nabla v|^2 \leq C   \left(\int  \eta_\delta |\nabla \eta_\delta|  |\nabla \Gamma|^{p-2} |\nabla v| +\int \eta_\delta^2 \frac{|\mathcal G|}{u}\right)
\end{eqnarray*}
in view of \eqref{conditionnablaH_phdthesis} (which follows by \eqref{1530} and $\|H_\lambda\|_\infty<+\infty$ when $p=N$) and then by the Young inequality
\begin{eqnarray} \label{01030}
\int  \eta_\delta^2  |\nabla \Gamma|^{p-2}  |\nabla v|^2 && \leq  C'   \left( \int |\nabla \eta_\delta|^2  |\nabla \Gamma|^{p-2} +\int \eta_\delta^2 \frac{|\mathcal G|}{u}\right)  \leq   C   \left( \int  |x|(\frac{|x|^2}{\delta +|x|^2})^{-\frac{(N-1)(p-2)}{2(p-1)}}  |\nabla \eta|^2  \right. \nonumber \\
&& \left. +\int (\frac{|x|^2}{\delta  +|x|^2})^{2-\frac{(N-1)(p-2)}{2(p-1)}}  \frac{\eta^2}{|x|} 
+\int |x| (\delta  +|x|^2)^{\frac{(N-1)(p-2)}{2(p-1)}}   \eta^2 \frac{|\mathcal G|}{u}\right)
\end{eqnarray}
for universal constants in $R$, $\delta$ and $c$.  Since $(\frac{|x|^2}{\delta  +|x|^2})^\alpha \leq C |x|^{-\max\{ -2\alpha,0\} }$, we have that 
\begin{equation} \label{01710}
\begin{array}{c}
|x| \displaystyle{ (\frac{|x|^2}{\delta +|x|^2})^{-\frac{(N-1)(p-2)}{2(p-1)}} \leq C \displaystyle |x|^{- \max\{\frac{(N-1)(p-2)}{p-1}-1,-1\} } \in L^1_{\hbox{loc}}(\mathbb{R}^N)} \\ 
\displaystyle{ (\frac{|x|^2 }{\delta +|x|^2})^{2-\frac{(N-1)(p-2)}{2(p-1)}} \frac{1}{|x|}\leq
C |x|^{-\max\{ \frac{(N-1)(p-2)}{p-1}-3,1\} } \in L^1_{\hbox{loc}}(\mathbb{R}^N) }
\end{array}
\end{equation} 
in view of $\frac{(N-1)(p-2)}{p-1}<N$. Since $\mathcal G=\pm \lambda R^p  \Gamma^{p-1} (x)[1+O(R^\frac{N-p}{p-1})]$ when $2\leq p<N$ in view of $\|H_\lambda\|_\infty<+\infty$, for $\lambda \not=0$  there holds $\Lambda \geq C R^{\frac{p}{p-1}}$ for some $C>0$ and all $R$ small in view of $q_0<\frac{N}{N-p}$, where $\Lambda$ is given by \eqref{definitionk}, and then
\begin{eqnarray}
\int  |x| (\delta +|x|^2)^\frac{(N-1)(p-2)}{2(p-1)}  \eta^2  \frac{|\mathcal{G}|}{u}& \leq& \frac{1}{\Lambda} \int  
|x| (\delta +|x|^2)^\frac{(N-1)(p-2)}{2(p-1)}  \eta^2 |\mathcal{G}|  \nonumber \\
&\leq& C  \int  |x|^{p+1-N}(\delta +|x|^2)^\frac{(N-1)(p-2)}{2(p-1)}  \eta^2.  \label{01445}
\end{eqnarray}
On the other hand, since $\mathcal G=\pm \lambda R^N  |\log R|^{N-1} [1+O(\frac{\log |x|}{\log R}) ]$ in $B_2(0)$ when $p=N$ thanks to $\|H_\lambda\|_\infty<+\infty$, for $\lambda \not=0$  there holds $\Lambda \geq C R^\frac{N}{N-1}  |\log R|$ for some $C>0$ and for all $R$ small and then
\begin{eqnarray}
\int  |x| (\delta +|x|^2)^\frac{N-2}{2}  \eta^2  \frac{|\mathcal{G}|}{u} \leq \frac{1}{\Lambda} \int  
|x| (\delta +|x|^2)^\frac{N-2}{2}  \eta^2 |\mathcal{G}|  
\leq C  \int  |x| |\log |x|| (\delta +|x|^2)^\frac{N-2}{2}  \eta^2.  \label{01445bis}
\end{eqnarray}
Since
\begin{eqnarray} \label{01653}
|x|^{p+1-N}|\log |x|| (\delta +|x|^2)^\frac{(N-1)(p-2)}{2(p-1)} \leq C |x|^{ p+1-N} |\log |x|| \in L^1_{\hbox{loc}}(\mathbb{R}^N)
\end{eqnarray}
when $\lambda \not=0$ in view of $p\geq 2$,  we can use \eqref{01710}, \eqref{01653} and the Lebesgue convergence Theorem in \eqref{01030} and \eqref{01445}-\eqref{01445bis} 
to get
\begin{eqnarray} \label{01715}
\int  \eta^2   |x|  |\nabla v|^2 \leq   C   \left( \int  |x| |\nabla \eta|^2+\int \frac{\eta^2}{|x|} +\underbrace{\int  |x|^{p-\frac{N-1}{p-1}} |\log |x||  \eta^2}_{\lambda \not= 0}  \right)
\end{eqnarray}
thanks to the Fatou convergence Theorem. Since $p-\frac{N-1}{p-1}>-1$ if $\lambda \not=0$ and
\begin{eqnarray*} 
\int_B |v - \bar{v}| &\leq& C' h \int_B |\nabla v|  \leq C' h (\int_B \frac{1}{|x|})^\frac{1}{2}(\int_B |x| |\nabla v|^2)^\frac{1}{2}  \nonumber \\
&\leq& Ch (\int_B \frac{1}{|x|})^\frac{1}{2}    \left( \int  |x| |\nabla \eta|^2+\int \frac{\eta^2}{|x|} +\underbrace{\int  |x|^{p-\frac{N-1}{p-1}} |\log |x|| \eta^2}_{\lambda \not=0} 
\right)^\frac{1}{2}
\end{eqnarray*} 
in view of \eqref{01715}, for $|x_0| < 3h$ one has that
\begin{eqnarray*} 
\int_B |v - \bar{v}| \leq C h^\frac{N+1}{2}\left(h^{N-1}+\underbrace{h^{p-\frac{N-1}{p-1}+N} |\log h|}_{\lambda \not=0}  \right)^\frac{1}{2}=O(h^N)
\end{eqnarray*} 
in view of $B_{\frac{3}{2}h}(x_0) \subset B_{5h}(0)$, while for $|x_0| \geq 3h$ there holds
\begin{eqnarray*} 
\int_B |v - \bar{v}| &\leq& C' \left[ h^2 (\int_B \frac{1}{|x|})(\int  |x| |\nabla \eta|^2)+h^{N+1}(h^{N-1}+|\log h| h^{ \min\{p-\frac{N-1}{p-1}+N , N \}} ) \right]^\frac{1}{2}\\
&\leq & C \left[ h^2 (\frac{h^N}{|x_0|})(|x_0| h^{N-2})+h^{2N}\right]^\frac{1}{2}=O(h^N)
\end{eqnarray*} 
in view of $\frac{3h}{2} \leq \frac{|x_0|}{2} \leq |x| \leq \frac{3}{2}|x_0|$  for all $x \in B_{\frac{3}{2}h}(x_0)$. The proof is complete. \end{proof}
We are now ready to establish an Harnack inequality for $\mathcal H=R^\frac{N-p}{p-1} (\pm H_\lambda(Rx)+c)$ when $\mathcal H \geq 0$, a crucial tool to establish the H\"older continuity of $H_\lambda$ at $0$.
\begin{theorem} \label{propositionharnack}
Let $1<p\leq N$ if $\lambda=0$ and $p\geq 2$ with $p> \frac{N}{2}$ if $\lambda \not=0$. Assume that $\mathcal H=R^\frac{N-p}{p-1} (\pm H_\lambda(Rx)+c) \geq  0$ in $B_2(0)$. Then there exist $R_0>0$ and $C>0$ so that 
\begin{equation} \label{harnackinequality}
\sup_{B_1(0)} \mathcal H \leq C(\inf_{B_1(0)} \mathcal H + \Lambda)  \end{equation}
for all $0<R \leq R_0$, where $\Lambda$ is given in \eqref{definitionk} in terms of $\mathcal G=\pm \lambda R^N G_\lambda^{p-1}(Rx)$
\end{theorem}
\begin{proof} Given $p_0>0$ to be specified below,  let us fix $0<p_1<p_0$ so that $\kappa^j p_1 \not=2-p,1$ for all $j\geq 0$. Consider first the case $\mu>0$ in Proposition \ref{propositionboundmathcalH} to get
\begin{equation} \label{01557}
\Phi_\rho(\kappa \mu, h_1) \leq  [\tilde C \mu^{\nu} (h_2-h_1)^{-\beta}]^\frac{1}{\mu} \Phi_\rho(\mu,h_2) 
\end{equation}
for all $\mu \not=2-p,1$ and for suitable $\nu,\beta \geq 0$, where $u= |\mathcal{H}| + \Lambda+ \epsilon \geq 0$. Starting from $p_1$ along $\mu_j=\kappa^j p_1 $ estimate \eqref{01557} with $1\leq h^j_{1}=1+ 2^{-(j+1)}<h^j_2=1+2^{-j}\leq 2$ gives
$$\Phi_\rho(\mu_{j+1}, h^j_1) \leq  [C (2^{\beta} \kappa^{\nu})^j]^\frac{1}{\kappa^j p_1} \Phi_\rho(\mu_j,h^j_2) $$
and then
\begin{equation} \label{harnackmaxinequality}
\sup_{B_1(0)} u \leq \lim_{j \to +\infty} \Phi_\rho(\mu_{j+1}, h^j_1)\leq C_1 \Phi_\rho(p_1,2), \quad C_1=C^\frac{\kappa}{p_1(\kappa-1)} 
(2^{\beta} \kappa^{\nu})^{\frac{1}{p_1}\sum_j \frac{j}{\kappa^j }}
\end{equation}
via an iteration argument as in the proof of Proposition  \ref{sup-estimate}.  Since $\rho>0$ in $B_1(0) \setminus \{0\}$, notice that
$$ \|u\|_{\infty, B_1(0)\setminus B_\epsilon(0)}  \leq \liminf_{\mu \to +\infty} \Phi_\rho(\mu,1)\leq \limsup_{\mu \to +\infty} \Phi_\rho(\mu,1)\leq \|u\|_{\infty, B_1(0)}$$
and then as $\epsilon \to 0$
\begin{equation}\label{251258}
\lim_{\mu \to +\infty} \Phi_\rho(\mu,1)= \|u\|_{\infty, B_1(0)}=\sup_{B_1(0)} u.
\end{equation}

\medskip \noindent Consider the case $\mu<0$ in Proposition \ref{propositionboundmathcalH}  to get
\begin{equation} \label{01603}
\Phi_\rho(\kappa \mu, h_1) \geq  [\tilde C |\mu|^{\nu} (h_2-h_1)^{-\beta}]^\frac{1}{\mu} \Phi_\rho(\mu,h_2) 
\end{equation}
for all $\mu \not=2-p$. Starting from $-p_1$ along $\mu_j=\kappa^j (-p_1)$, one can use estimate \eqref{01603} with $h^j_{1}$ and $h^j_2$ to get 
$$\Phi_\rho(\mu_{j+1}, h^j_1) \geq  [C (2^{\beta} \kappa^{\nu})^j]^{-\frac{1}{\kappa^j p_1}} \Phi_\rho(\mu_j,h^j_2)$$ 
and then, arguing as we did to show \eqref{251258}, one deduces that
\begin{equation} \label{harnackmininequality}
\inf_{B_1(0)} u  \geq \lim_{j \to +\infty} \Phi_\rho(\mu_{j+1}, h^j_1)\geq C_2 \Phi_\rho(-p_1,2), \quad C_2=C^{-\frac{\kappa}{p_1(\kappa-1)}} 
(2^{\beta} \kappa^{\nu})^{-\frac{1}{p_1}\sum_j \frac{j}{\kappa^j }},
\end{equation}
in view of $\mu_j \to -\infty$ as $j \to +\infty$.

\medskip \noindent Assume now $\mathcal H\geq 0$ in $B_2(0)$. Let us finally use Proposition \ref{lemmaJN} to compare \eqref{harnackmaxinequality} and \eqref{harnackmininequality}. Indeed,  as a consequence of John-Nirenberg Lemma (see  Lemma 7 in \cite{serrin}), Proposition \ref{lemmaJN} shows the existence of $p_0>0$ so that
\begin{equation*}
\left(\int_{B_2(0)} \rho e^{p_0 v} \int_{B_2(0)} \rho e^{-p_0 v}  \right)^{\frac{1}{p_0}} \leq 
\|\rho \|_{\infty,B_2(0)}^\frac{2}{p_0} \left(\int_{B_2(0)} e^{p_0 v} \int_{B_2(0)} e^{-p_0 v}  \right)^{\frac{1}{p_0}} \leq C_3 
\end{equation*}
for some universal $C_3>0$, or equivalently
\begin{equation} \label{harnackPhiinequality}
\Phi_\rho(p_0,2) \leq C_3 \Phi_\rho (-p_0,2)
\end{equation}
in terms of $u=e^v=\mathcal H+\Lambda+\epsilon$.  The use of \eqref{harnackPhiinequality} along with \eqref{harnackmaxinequality} and \eqref{harnackmininequality} gives
\begin{eqnarray*}
\sup_{B_1(0)} u  \leq C_1 \Phi_\rho(p_1,2) \leq C_1' \Phi_\rho(p_0,2) \leq C_1' C_3  \Phi_\rho (-p_0,2) \leq C_1' C_3' \Phi_\rho (-p_1,2) \leq \frac{C_1' C_3'}{C_2} \inf_{B_1(0)}u
\end{eqnarray*}
thanks to the H\"older estimate in view of $p_1<p_0$ and $\rho \in L^\infty(B_2(0))$. Since $u=\mathcal H+\Lambda +\epsilon$, one then deduces
\begin{eqnarray*}
\sup_{B_1(0)} \mathcal H  \leq C( \inf_{B_1(0)} \mathcal H +\Lambda+\epsilon )
\end{eqnarray*}
for some $C>0$ and \eqref{harnackinequality} follows by letting $\epsilon \to 0$.
\end{proof}
In particular,  for $p\geq 2$ we have the following a-priori $L^\infty-$estimate.
\begin{corollary} \label{boundHepsilon}
Let $2\leq p\leq N$. Given $M>0$ and $p_0 \geq 1$ there exists $C>0$ so that
\begin{equation} \label{corollarylocalboundH3}
\| h+c \|_{\infty,B_R(0)} \leq C (R^{-\frac{N}{p_0}} \| h+c \|_{p_0,B_{2R}(0)} +R^\frac{pq_0-N}{q_0(p-1)} \|f \|^\frac{1}{p-1}_{q_0,B_{2R}(0)}) 
\end{equation}
for all $\epsilon^{p-1}\leq R\leq R_0=\frac{1}{4}\hbox{dist}(0,\partial \Omega)$ and all solution $h$ to 
$$-\Delta_p (u+h)+\Delta_p  u= f \qquad \text{in } \Omega \setminus \{0\}$$
so that $\nabla h \in L^{\bar q}(\Omega)$, $\frac{|x|^\frac{1}{p-1}}{M(\epsilon^p+|x|^\frac{p}{p-1})^\frac{N}{p}}\leq  |\nabla u| \leq M|\nabla \Gamma|$ for some $\epsilon>0$ and $|c|+\|h \|_\infty+ \| f \|^\frac{1}{p-1}_{q_0} \leq M$ for some $q_0>\frac{N}{p}$.
\end{corollary}
\begin{proof}
Set $\mathcal H(x)=R^\frac{N-p}{p-1} (h(Rx)+c)$, $0<R<2R_0$. We have that $\mathcal H \in L^\infty(B_2(0))$ solves \eqref{mathcalHproblem} with $\mathit \Gamma=R^\frac{N-p}{p-1} u(Rx)$, $\mathcal G=R^N f(Rx)$ and satisfies $\nabla \mathcal H \in L^{\bar q}(B_2(0))$. Since $\|\mathcal H\|_{\infty,B_2(0)} \leq 2M R^\frac{N-p}{p-1}$ and
\begin{equation} \label{01600}
\|\mathcal G\|^\frac{1}{p-1}_{q_0,B_2(0)}=R^{\frac{N(q_0-1)}{q_0(p-1)}}  \|f\|^\frac{1}{p-1}_{q_0,B_{2R}(0)} \leq M R^{\frac{N(q_0-1)}{q_0(p-1)}} ,
\end{equation}
we have that
$$\|\mathcal H\|_{\infty,B_2(0)}+\|\mathcal G\|^\frac{1}{p-1}_{q_0,B_2(0)} \leq \tilde M$$
for some $\tilde M$ and all $0<R\leq R_0$. Since
$$\frac{|x|^\frac{1}{p-1}}{M 2^{\frac{N}{p-1}+\frac{N}{p}}}
\leq \frac{|x|^\frac{1}{p-1}}{M ((\epsilon^{p-1}R^{-1})^\frac{p}{p-1}+|x|^\frac{p}{p-1})^\frac{N}{p}}
\leq  |\nabla \mathit \Gamma| \leq M |\nabla \Gamma|$$
in $B_2(0)$ for $\epsilon^{p-1}\leq R\leq R_0$, Proposition \ref{propositionboundmathcalH} gives the validity of \eqref{lemma1mathcalH3} for all $\mu \not=0$ and we can argue as in \eqref{harnackmaxinequality} to get
\begin{equation} \label{01617}
\sup_{B_1(0)} u \leq C_1 \Phi_1(p_1,2)
\end{equation}
for a given $0<p_1<p_0$ so that $\kappa^j p_1 \not= 1$ for all $j \in \mathbb{N}$, where $u=|\mathcal H|+\|\mathcal G\|^\frac{1}{p-1}_{q_0,B_2(0)}+\epsilon'$. Since $\Phi_1(p_1,2) \leq |B_2(0)|^\frac{p_0-p_1}{p_0 p_1} \Phi_1(p_0,2)$ by H\"older estimate, by \eqref{01617} we deduce that
\begin{eqnarray} \label{01617bis}
\| h+c \|_{\infty,B_R(0)}&=& R^{-\frac{N-p}{p-1}}\sup_{B_1(0)} |\mathcal H| \leq C' R^{-\frac{N-p}{p-1}} (
\|\mathcal H\|_{p_0,B_2(0)}+\|\mathcal G\|^\frac{1}{p-1}_{q_0,B_2(0)}+\epsilon') \nonumber \\
&\leq &
C \left(R^{-\frac{N}{p_0}} \| h+c \|_{p_0,B_{2R}(0)} +R^\frac{pq_0-N}{q_0(p-1)} \|f \|^\frac{1}{p-1}_{q_0,B_{2R}(0)}+\epsilon' R^{-\frac{N-p}{p-1}} \right) 
\end{eqnarray}
in view of \eqref{01600} and
$$ \|\mathcal H\|_{p_0,B_2(0)}=R^\frac{N-p}{p-1} R^{-\frac{N}{p_0}} \|h +c\|_{p_0,B_{2R}(0)}.$$
Letting $\epsilon' \to 0$ in \eqref{01617bis} we deduce the validity of \eqref{corollarylocalboundH3} and the proof is complete.
\end{proof}
Finally, let us discuss the H\"older regularity of $H_\lambda$ at the pole $0$. Given $\Lambda$ in \eqref{definitionk} in terms of $\mathcal G=\pm \lambda R^N G_\lambda^{p-1}(Rx)$, let us re-write the Harnack inequality \eqref{harnackinequality} for $\mathcal H=R^\frac{N-p}{p-1} (\pm H_\lambda(Rx)+c)\geq 0$ in $B_{2R}(0)$ as
\begin{equation} \label{harnackinequality'}
  \sup_{B_R(0)} (\pm H_\lambda+c) \leq C \left(  \inf_{B_R(0)} (\pm H_\lambda+c) +  R^\sigma \right) 
\end{equation}
for all $0<R\leq R_0$, in view of \eqref{01600} with $f=\pm \lambda G_\lambda^{p-1}$.  Since we assume $p\geq 2$ with $p>\frac{N}{2}$ if $\lambda \not=0$, notice that $\sigma= \frac{pq_0-N}{q_0(p-1)}>0$ when $\lambda \not=0$ in view of \eqref{01609} with $q_0>\frac{N}{p}$, while the term $R^\sigma$ is not present when $\lambda=0$. In this second case, we can assume $\sigma \in (0,+\infty)$. 

\medskip \noindent We are now in position to follow the argument in \cite{serrin} and establish the following H\"older property. 
\begin{theorem} \label{holdercontinuityatthepole}
Let $1<p\leq N$ if $\lambda=0$ and $p\geq 2$ with $p>\frac{N}{2}$ if $\lambda \not= 0$. Then $H_\lambda \in C(\bar \Omega)$ and there exists $C>0$ such that
\begin{equation} \label{holdercontinuity}
|H_\lambda(x)- H_\lambda(0)| \leq C |x|^\alpha \quad \forall \ x \in \Omega 
\end{equation}
for some $ \alpha \in (0,1)$.
\end{theorem}
\begin{proof}
Setting $M(R)= \displaystyle \sup_{B_R(0)} H_\lambda$ and $\mu(R)= \displaystyle \inf_{B_R(0)} H_\lambda$ for $R>0$, we claim that the oscillation $\omega(R)=M(R)-\mu(R)$ of $H$ in $B_R(0)$ satisfies
\begin{equation} \label{oscillationat0_holdercontinuity}
\omega(R) \leq C_0 R^\alpha 
\end{equation}
for all  $0<R\leq R_0$, for some $\alpha, C_0, R_0>0$. 

\medskip \indent Indeed, apply \eqref{harnackinequality'} on $B_\frac{R}{2}(0)$ either with $c=M(R)$ and the $-$ sign or with $c=-\mu(R)$ and the $+$  sign to get
\begin{equation} \label{01701}
M(R)-\mu'(R)\leq C[M(R)-M'(R)]+C R^\sigma, \quad M'(R)-\mu(R) \leq C [\mu'(R)-\mu(R)]+C R^{\sigma}
\end{equation}
for all $0<R \leq 2R_0$, where $M'(R)=M(\frac{R}{2})$ and $\mu'(R)=\mu(\frac{R}{2})$. By adding the two inequalities in \eqref{01701} we get that
\begin{equation} \label{holder3_phdthesis}
\omega(\frac{R}{2}) \leq \theta \omega(R)+C_0 R^{\sigma}
\end{equation}
for all $0<R \leq 2R_0$, where $\theta =\frac{C-1}{C+1}<1$ and $C_0=\frac{2C}{C+1}$. If $\theta\leq 0$, then \eqref{holder3_phdthesis} implies the validity of \eqref{oscillationat0_holdercontinuity} with $\alpha=\sigma>0$ for all $0<R\leq R_0$ and some $C_0>0$ . In the case $\theta >0$, for $S\geq 2$ \eqref{holder3_phdthesis} gives that
\begin{equation*}
\omega (\frac{R}{S})\leq \omega (\frac{R}{2}) \leq \theta (\omega(R) + \tau R^\sigma), \quad 0 <R \leq R_0,
\end{equation*}
for some $\tau>0$ and an iteration starting from $r=R_0$ leads to
\begin{equation} \label{holder5_phdthesis}
\omega(\frac{R_0}{S^j}) \leq  \theta^j [ \omega(R_0) + \tau R_0^\sigma \sum_{k=0}^{j-1} (\theta S^\sigma)^{-k} ].
\end{equation}
Since $\theta \in (0,1)$ and $\sigma>0$ in \eqref{holder3_phdthesis} can be taken smaller than $1$, the choice $S=(\frac{2}{\theta})^\frac{1}{\sigma} \geq 2$ is admissible in \eqref{holder5_phdthesis} yielding
\begin{equation} \label{holder6_phdthesis}
\omega(\frac{R_0}{S^j}) \leq \theta^j (\omega(R_0) +  2 \tau R_0^\sigma).
\end{equation}
Given $0<R \leq \frac{R_0}{S}$, let $j_0 \geq 1$ be so that $\frac{R_0}{S^{j_0+1}} < R \leq \frac{R_0}{S^{j_0}}$ and by \eqref{holder6_phdthesis} we have 
\begin{equation} \label{holder7_phdthesis}
\omega(R) \leq \omega(\frac{R_0}{S^{j_0}}) \leq \theta^{j_0} (\omega(R_0) + 2 \tau R_0^\sigma)  \leq C \theta^{j_0}
\end{equation}
with $C=\omega(R_0)+2\tau R_0^\sigma$. Setting $\gamma =-\frac{\log \theta}{\log 2}>0$, then $\theta=2^{-\gamma}=S^{-\alpha}$ with $\alpha= \frac{\sigma \gamma}{\gamma +1} \in (0,1)$ and \eqref{holder7_phdthesis} implies
\begin{equation*}
\omega(R) \leq C (\frac{S}{R_0})^\alpha R^\alpha 
\end{equation*}
for all $0<R \leq R_0 2^{-\frac{\gamma + 1}{\sigma}}$, and \eqref{oscillationat0_holdercontinuity} is established in this case too.

\medskip \noindent Since \eqref{oscillationat0_holdercontinuity} gives that $\displaystyle \lim_{R \to 0} \omega(R)=0$, we deduce that $H_\lambda \in C(\bar \Omega)$ in view of $G_\lambda \in C^{1,\beta}(\bar \Omega \setminus \{0 \})$ by elliptic regularity theory \cite{dib,lieberman,serrin,tolksdorf}. Setting $R=|x|$, \eqref{oscillationat0_holdercontinuity} implies
$$|H_\lambda(x)-H_\lambda(0)| \leq \omega(R )\leq C_0|x|^{\alpha}$$
for all $x \in B_{R_0}(0)$. Since  \eqref{holdercontinuity} clearly holds in $\Omega \setminus B_{R_0}(0)$ in view of the boundedness of $H_\lambda$, we get the validity of \eqref{holdercontinuity}  in the whole $\Omega$ and the proof is complete.
\end{proof}

\bibliographystyle{plain}

\begin{thebibliography}{99}

\bibitem{AgPe} 
J.A. Aguilar Crespo, I. Peral, \emph{Blow-up behavior for solutions of $-\Delta_{N} u=V(x)e^{u}$ in bounded domains in $\mathbb{R}^{N}$},  Nonlinear Anal.  {\bf 29} (1997),  no. 4, 365--384.

\bibitem{AnEs2}  S. Angeloni, P. Esposito, \emph{The quasi-linear Brezis-Nirenberg problem in low dimensions},  preprint arXiv.

\bibitem{BBGGPV} P. B\'{e}nilan, L. Boccardo, T. Gallou\"{e}t, R. Gariepy, M. Pierre, J.L. Vazquez, \emph{An $L^{1}$-theory of existence and uniqueness of solutions of nonlinear elliptic equations}, Ann. Scuola Norm. Sup. Pisa {\bf 22} (1995), no. 2, 241--273.

\bibitem{boccardogallouet1}
L. Boccardo, T. Gallou\"{e}t, \emph{Nonlinear elliptic and parabolic equations involving measure data}, J. Funct. Anal. {\bf 87} (1989), no. 1, 149--169.

\bibitem{boccardogallouet2}
L. Boccardo, T. Gallou\"{e}t, \emph{Nonlinear elliptic equations with right-hand side measures}, Comm. Partial Differential Equations {\bf 17} (1992), no. 3-4, 641--655.

\bibitem{brezisnirenberg}
H. Brezis, L. Nirenberg, \emph{Positive solutions of nonlinear elliptic equations involving critical Sobolev exponents}, Comm. Pure Appl. Math. {\bf 36} (1983), 437--477.

\bibitem{CKN}
L. Caffarelli, R. Kohn, L. Nirenberg, \emph{First order interpolation inequalities with weights}, Composition Math. {\bf 53} (1984), 259--275.

\bibitem{CuTa}
M. Cuesta, P. Takac, \emph{A strong comparison principle for the Dirichlet $p-$laplacian}, Reaction diffusion systems, Trieste (1995), 79--87.

\bibitem{DMOP} G.  Dal Maso,  F.  Murat,  L. Orsina,  A. Prignet,  \emph{Renormalized solutions of elliptic equations with general measure data},  Ann. Scuola Norm. Sup. Pisa {\bf 28} (1999), no. 4, 741–808. 

\bibitem{diazsaa}
J.I. D\`iaz, J.E. Saa, \emph{Existence et unicit\'e de solutions positives pour certaines \'equations elliptiques quasilin\'eaires}, C. R. Acad. Sci. Paris S\'er. I Math. {\bf 305} (1987), 521--524.

\bibitem{dib} E. Dibenedetto, \emph{$C^{1+\alpha}$ local regularity of weak solutions of degenerate elliptic equations}, Nonlinear Anal. {\bf 7} (1983), no. 8, 827--850.

\bibitem{DHM1} G. Dolzmann,  N. Hungerb\"{u}hler,  S.  M\"{u}ller,  \emph{The p-harmonic system with measure-valued right hand side},  Ann. Inst. H. Poincar\'e Anal. Non Lin\'eaire {\bf 14} (1997), no. 3, 353–364.

\bibitem{DHM2} G. Dolzmann,  N. Hungerb\"{u}hler,  S.  M\"{u}ller,  \emph{Non-linear elliptic systems with measure-valued right hand side},  Math. Z.  {\bf 226} (1997), no. 4, 545–574.

\bibitem{DHM3} G. Dolzmann,  N. Hungerb\"{u}hler,  S.  M\"{u}ller,  \emph{Uniqueness and maximal regularity for nonlinear elliptic systems of n-Laplace type with measure valued right hand side}, J. Reine Angew. Math.  {\bf 520} (2000), 1–35.

\bibitem{Dru}
O. Druet, \emph{Elliptic equations with critical Sobolev exponents in dimension $3$}, Ann. Inst. H. Poincar\'e Anal. Non Lin\'eaire {\bf 19} (2002), no. 2, 125--142.

\bibitem{Esp} P. Esposito, \emph{On some conjectures proposed by Ha\"{i}m Brezis}, Nonlinear Anal. {\bf 54} (2004), no. 5, 751--759.
 
\bibitem{Esp1}
P. Esposito, \emph{A classification result for the quasi-linear Liouville equation}, Ann. Inst. H. Poincar\'e Anal. Non Lin\'eaire {\bf 35} (2018), no. 3, 781--801.

\bibitem{Eva}
L.C. Evans, \emph{A new proof of local $C^{1,\alpha}$ regularity for solutions of certain degenerate elliptic p.d.e.}, J. Differential Equations {\bf 45} (1982), no. 3, 356--373. 

\bibitem{takac}
J. Fleckinger-Pell\`e, J. Hern\`andez, P. Takac, F. De Th\'elin, \emph{Uniqueness and positivity for solutions of equations with the $p$-laplacian}, Reaction diffusion systems, Trieste (1995), 141--155.

\bibitem{fleckingertakac}
J. Fleckinger-Pell\`e, P. Takac, \emph{Uniqueness of positive solutions for nonlinear cooperative systems with the $p$-laplacian}, Indiana Univ. Math. J. {\bf 43} (1994), 1227--1253.

\bibitem{GIS} L. Greco,  T.  Iwaniec, C. Sbordone,  \emph{Inverting the p-harmonic operator},  Manuscripta Math. {\bf 92} (1997), no. 2, 249–258.

\bibitem{gueddaveron2}
M. Guedda, L. Veron, \emph{Quasilinear elliptic equations involving critical Sobolev exponents}, Nonlinear Anal. {\bf 13} (1989), no. 8, 879--902.

\bibitem{HKM} J. Heinonen, T.  Kilpel\"{a}inen, O. Martio, \emph{Nonlinear potential theory of degenerate elliptic equations}. Oxford Mathematical Monographs (1993). The Clarendon Press, Oxford University Press, New York.

\bibitem{kichenassamyveron}
S. Kichenassamy, L. Veron, \emph{Singular solutions of the $p$-Laplace equation}, Math. Ann. {\bf 275} (1986), 599--616.

\bibitem{KuMi} T. Kuusi, G. Mingione, \emph{Guide to nonlinear potential estimates}, Bull. Math. Sci. {\bf 4} (2014), no. 1, 1--82. 

\bibitem{LaUr}
O.A. Ladyzhenskaya, N.N. Ural'ceva, \emph{Linear and quasilinear elliptic equations}. Academic Press, New York-London 1968.

\bibitem{lewis}
J.L. Lewis, \emph{Regularity of the derivatives of solutions to certain degenerate elliptic equations}, Indiana Univ. Math. J. {\bf 32} (1983), 849--858. 

\bibitem{lieberman}
G.M. Lieberman, \emph{Boundary regularity for solutions of degenerate elliptic equations}, Nonlinear Anal. {\bf 12} (1988), no. 11, 1203--1219.

\bibitem{manfredi}
J.J. Manfredi, \emph{Isolated singularities of $p$-harmonic functions in the plane}, SIAM J.  Math.  Anal.  {\bf 22} (1991), no. 2, 424--439.

\bibitem{Mur} F.  Murat, \emph{Soluciones renormalizadas de EDP elipticas no lineales},  Publications du Laboratoire d'Analyse Num\'erique R93023, (1993).

\bibitem{orsina} L.  Orsina, \emph{Solvability of linear and semilinear eigenavlue problems with $L^1$ data},  Rend.  Sem.  Mat.  Univ.  Padova {\bf 90} (1993),  207--238.

\bibitem{serrin}
J. Serrin, \emph{Local behaviour of solutions of quasilinear equations}, Acta Math. {\bf 111} (1964), 247--302.

\bibitem{serrin65}
J. Serrin, \emph{Isolated singularities of solutions of quasilinear equations}, Acta Math. {\bf 113} (1965), 219--240.

\bibitem{tolksdorf}
P. Tolksdorf, \emph{Regularity for a more general class of quasilinear elliptic equations}, J. Differential Equations {\bf 51} (1984), 126--150.

\bibitem{Ura}
N.N. Ural'ceva, \emph{Degenerate quasilinear elliptic systems}, Sem. Math. V.A. Steklov, Math. Inst. Leningrad {\bf 7} (1968), 184--222.
 
 \bibitem{vazquez}
J.L. Vazquez, \emph{A strong maximum principle for some quasilinear elliptic equations}, Appl. Math. Optim. {\bf 12} (1984), 191--202.

\end{thebibliography}

\end{document}